\documentclass[11pt]{amsart}
\usepackage{amsfonts}
\usepackage{amssymb}
\usepackage{amsmath}
\usepackage{amsthm}
\usepackage[numbers]{natbib}
\usepackage{graphicx}
\usepackage{enumitem}
\usepackage{braket}
\usepackage[pdftex,plainpages=false,colorlinks,hyperindex,bookmarksopen,linkcolor=red,citecolor=blue,urlcolor=blue]{hyperref}
\usepackage{mathrsfs}
\usepackage{xcolor}
\usepackage{epstopdf}
\usepackage{braket}
\usepackage{color}

\calclayout
\makeatletter
\def\author@andify{  \nxandlist {\unskip ,\penalty-1 \space\ignorespaces}    {\unskip {} \@@and~}    {\unskip \penalty-2 \space \@@and~}}
\makeatother
\pdfoutput=1
\bibpunct{[}{]}{;}{n}{,}{,}
\newtheorem{theorem}{Theorem}[section]

\newtheorem{coro}{Corollary}[section]
\newtheorem{definition}{Definition}[section]
\theoremstyle{remark}
\newtheorem{remark}{Remark}[section]

\newtheorem{lem}{Lemma}[section]

\newcommand{\be}{\begin{eqnarray}}
\newcommand{\ee}{\end{eqnarray}}

\newcommand{\C}{\mathbb{C}}

\newcommand{\func}[1]{\operatorname{#1}}

\numberwithin{equation}{section}
\allowdisplaybreaks

\begin{document}
\title{Non-Gaussian measures in infinite dimensional spaces: the Gamma-grey noise}

\author[Beghin]{Luisa Beghin$^1$}
\address{${}^1$ Sapienza University of Rome. P.le Aldo Moro, 5, Rome, Italy}
\email{luisa.beghin@uniroma1.it}

\author[Cristofaro]{Lorenzo Cristofaro$^2$}
\address{${}^2$ Sapienza University of Rome. P.le Aldo Moro, 5, Rome, Italy}
\email{lorenzo.cristofaro@uniroma1.it}

\author[Gajda]{Janusz Gajda$^3$}
\address{${}^3$ University of Warsaw. Dluga 44/50 00-241, Warsaw, Poland}
\email{jgajda@wne.uw.edu.pl}

\keywords{Incomplete gamma function, Completely monotone functions, Grey
noise, Hitting times, Fractional Brownian motion, Elliptically contoured
measures}
\subjclass[2020]{Primary 60G20; Secondary 60G22, 33B20, 60H40.}
\date{\today }

\begin{abstract}
In the context of non-Gaussian analysis, Schneider \cite{SCH} introduced
grey noise measures, built upon Mittag-Leffler functions;
analogously, grey Brownian motion and its generalizations were constructed
(see, for example, \cite{MUR2}, \cite{BOC}, \cite{DAS2}, \cite{DAS}). In this paper, we construct and
study a new non-Gaussian measure, by means of the incomplete-gamma function
(exploiting its complete monotonicity). We label this measure Gamma-grey noise and we prove, for it,
the existence of Appell system. The related generalized processes, in the infinite dimensional setting, 
are also defined and, through the use of the Riemann-Liouville fractional operators, the (possibly tempered)
Gamma-grey Brownian motion is consequently introduced.
A number of different characterizations of these processes are also provided, together with
the integro-differential equation satisfied by their transition densities.
They allow to model anomalous diffusions, mimicking the
procedures of classical stochastic calculus. 
\end{abstract}

\maketitle

\bigskip

\section{Introduction}

Non-Gaussian analysis has been introduced in the Nineties (see, for example,
\cite{ALB}, \cite{ALB2}, \cite{BER}), in order to extend the standard
infinite-dimensional (or white noise) constructions; see also \cite{OKS}. In particular, grey
noise has been defined for the first time by Schneider in \cite{SCH},
exploiting the complete monotonicity property of the Mittag-Leffler
function. Consequently, grey Brownian motion was also introduced in the same
paper and studied in \cite{SCH2}, allowing to model anomalous diffusions by
mimicking the classical procedures. These models represent a family of
(self-similar) stochastic processes, with stationary increments, which
includes, as special cases, both standard and fractional Brownian motion.
\newline
A further generalization (generalized grey Brownian motion, hereafter ggBm)
in due to \cite{MUR2}; it is also proved in \cite{MUR} that its marginal
density function is the fundamental solution of a stretched time-fractional
master equation.\newline
The ggBm, denoted by $B_{\alpha }^{\beta }:=\{B_{\alpha }^{\beta }(t),t>0\}$%
, for any $\alpha ,\beta \in (0,1]$, is characterized by the following $n$
-times characteristic function: for $\xi _{j}\in \mathbb{R},\text{ }%
j=1,...,n,$ and $\text{ }0\leq t_{1}\leq ...\leq t_{n}<\infty $
\begin{equation}
\mathbb{E}e^{i\sum_{j=1}^{n}\xi _{j}B_{\alpha }^{\beta }(t_{j})}=E_{\beta
}\left( -\frac{1}{2}\sum_{j,k=1}^{n}\xi _{j}\xi _{k}\gamma _{\alpha
}(t_{j},t_{k})\right) ,  \label{gbm}
\end{equation}%
where $\gamma _{\alpha }(t_{j},t_{k}):=t_{k}^{\alpha }+t_{j}^{\alpha
}-|t_{k}-t_{j}|^{\alpha }$ and $E_{\beta }\left( x\right) $ is the
Mittag-Leffler function $E_{\beta }\left( x\right): =\sum_{j=0}^{\infty
}x^{j}/\Gamma (\beta j+1),$ $x\in \mathbb{R}$ (see Appendix A, for details
on the Mittag-Leffler function in a more general definition). \newline
The link with Ornstein-Uhlenbeck process is explored, by means of the
stochastic calculus tools, in \cite{BOC}; this is made possible by the
representation of ggBm as a product of a fractional Brownian motion and an
independent random variable (with distribution depending on $\beta ).$

It is easy to see from (\ref{gbm}) that, for $\beta =1,$ the process $%
B_{\alpha }^{\beta }$ reduces to fractional Brownian motion with Hurst
parameter $H=\alpha /2$; for $\alpha =\beta ,$ it is called grey Brownian
motion (see \cite{SCH}); on the other hand, for $\alpha =\beta =1,$ it
coincides with standard Brownian motion. A slightly different construction
of the process, by means of the so-called Mittag-Leffler analysis, can be
found in \cite{GRO} and \cite{GRO2}. Finally, stochastic differential
equations driven by ggBm are studied in \cite{DAS}.

Our aim in this paper is to define, analogously to ggBm, another class of
processes that includes, as special cases both standard and fractional
Brownian motion. Our starting point is a result proved in \cite{BEG}, i.e.
that the upper incomplete gamma function $\Gamma (\rho
,x):=\int_{x}^{+\infty }e^{-w}w^{\rho -1}dw$ is completely monotone and that
the inverse Laplace transform of
\begin{equation}
\mathbb{\varphi }\mathbb{(\eta )=}\Gamma (\rho ,\eta ),\qquad \eta \geq 0,
\label{ll2}
\end{equation}%
reads%
\begin{equation}
f_{\rho }(y):=\mathcal{L}^{-1}\left\{ \mathbb{\varphi }\mathbb{(\cdot )}%
;y\right\} =1_{y>1}{\LARGE G}_{1,1}^{1,0}\left[ \left. \frac{1}{y}%
\right\vert
\begin{array}{c}
2 \\
1+\rho%
\end{array}%
\right] ,\qquad \rho \in (0,1],  \label{kl}
\end{equation}%
where ${\LARGE G}_{p,q}^{m,n}\left[ \left. \cdot \right\vert \;\right] $ is
the Meijer G-function (see (\ref{mei}), in Appendix A).

Moreover, (\ref{kl}) is a proper density function, up to the constant $%
1/\Gamma (\rho )$. We introduce here a tempering factor $\theta $, for $%
\theta \geq 0$, i.e. we will refer to $\Gamma (\rho ,\theta +\eta )$, for $%
\eta \geq 0$; the tempering is necessary in order to ensure finite moments
to the corresponding measure. The complete monotonicity of $\Gamma (\rho
,\theta +\cdot )$ easily follows. Once normalized by $\Gamma (\rho ,\theta )$%
, it will be used to define the characteristic functional of a measure, that
we will call $\Gamma $-grey measure.

In Section 2 we define the $\Gamma $-grey measure both on the finite and
infinite dimensional spaces, computing its moments and discussing the
existence of the Appell system \cite{KON1}. These steps are necessary in
order to extend the non-Gaussian analysis to the $\Gamma $-grey noise space
and require some well-known preliminary results on complexification and
holomorphic property in infinite dimensional spaces, that we present in the
Appendix (together with some formulae on special functions).

On the $\Gamma $-grey noise space, in Section 3, we define the tempered $%
\Gamma $-grey Brownian motion $B_{\alpha ,\rho }^{\theta }:=\{B_{\alpha
,\rho }^{\theta }(t),t>0\}$, for any $\alpha ,\rho \in (0,1],$ $\theta \geq
0,$ as generalized process, by means of the fractional operator $%
M_{-}^{\alpha /2}$, defined below (in terms of Riemann-Liouville derivative
or integral, depending on the values of $\alpha $). The tempering parameter $%
\theta $ is introduced in order to ensure finiteness of moments, while the
parameter $\rho $ (of the upper-incomplete gamma function) represents the
``distance" from the white noise setting: for $\rho =1$ (for any $\theta )$,
the process $B_{\alpha ,\rho }^{\theta }$ coincides with the fractional
Brownian motion with Hurst parameter $H=\alpha /2$, while, if we also put $%
\alpha =1$, we obtain the standard Brownian motion $B$.\ We prove that, in
the $n$-dimensional space, the tempered $\Gamma $-grey Brownian motion can
be fully characterized as a product of a fractional Brownian motion and an
independent random variable, defined on $[1,+\infty )$ and with distribution
depending on $\rho $ and $\theta .$ This factorization permits us to
interpret the distribution of the process as a Gaussian variance mixture
and, moreover, it is suitable for path-simulating purposes.

In Section 5 we discuss the time-change representation of this process
(which is valid for its one-dimensional distribution and for $\theta =0$),
i.e. the following equality in distribution
\begin{equation}
B_{\alpha ,\rho }(t)\overset{d}{=}B(Y_{\rho }(t^{\alpha })),\qquad t\geq 0.
\label{re}
\end{equation}%
Here we put, for simplicity, $B_{\alpha ,\rho }:=B_{\alpha ,\rho }^{0}$
while $Y_{\rho }:=\left\{ Y_{\rho }(t),t\geq 0\right\} $ is a stochastic
process, independent of the Brownian motion $B$, taking values in $%
[t,+\infty ),$ for any $t.$ Moreover, we derive, in the same
setting, the differential equations satisfied by its
characteristic function and by its transition density. Unlike what
happens in the case of the ggBm, the time-stretching parameter in
(\ref{re}) depends only on $\alpha $, while does not involve $\rho
$.

\section{The Gamma-grey noise}

We define the $\Gamma $-grey noise starting from the $n$-dimensional
Euclidean space, in analogy with the construction of the grey noise (see
\cite{SCH}) and the generalized grey noise (see \cite{MUR2}). In particular,
we will follow the slightly different approach introduced by \cite{GRO}. By
the complete monotonicity of $\Gamma \left( \rho ,\theta + \cdot \right)$,
for $\theta \geq 0$ and applying the Bernstein's theorem, there exists a
unique probability measure $\mu _{\rho ,\theta }$ on $[0,+\infty )$ such
that
\begin{equation}  \label{1dim}
\frac{\Gamma (\rho ,\theta + \eta )}{\Gamma (\rho, \theta )}%
=\int_{0}^{+\infty }e^{-\eta s}d\mu _{\rho ,\theta }(s),\qquad \eta ,\theta
\geq 0.
\end{equation}
%and $\mu _{\rho ,\theta }$ is absolutely continuous with respect to the
%Lebesgue measure on $(0,+\infty )$, with density (\ref{kl}), for $\lambda=1.$}
Moreover, the mapping
\begin{equation}
\mathbb{R}^{n}\ni \xi \rightarrow \frac{\Gamma (\rho ,\theta + \frac{1}{2}%
\left\langle \xi ,\xi \right\rangle _{euc})}{\Gamma (\rho ,\theta )}\in
\mathbb{R},  \label{cf}
\end{equation}
(where $\left\langle \cdot ,\cdot \right\rangle _{euc}$ denotes the
Euclidean scalar product on $\mathbb{R}^{n})$ is a characteristic function.

\begin{definition}
Let $n \in \mathbb{N}$, $\rho \in (0,1]$ and $\theta>0$. The $n$-dimensional
$\Gamma$-grey measure is the unique probability measure $\nu _{\rho ,\theta
}^{n}$ on $(\mathbb{R}^{n}, \mathcal{B}(\mathbb{R}^{n}))$ that satisfies:
\begin{equation}
\int_{\mathbb{R}^{n}}e^{i\left\langle x,\xi \right\rangle }d\nu _{\rho
,\theta }^{n}(x)=\frac{\Gamma (\rho ,\theta + \frac{1}{2}\left\langle \xi
,\xi \right\rangle _{euc} )}{\Gamma (\rho ,\theta )} ,\qquad \xi \in \mathbb{%
R}^{n}.  \label{yy3}
\end{equation}%
We define $\Phi _{\rho ,\theta }(\xi )$ as its characteristic function and
we call $(\mathbb{R}^{n}, \mathcal{B}(\mathbb{R}^{n}), \nu^n_{\rho,\theta})$
the $n$-dimensional $\Gamma$-grey space.
\end{definition}

\begin{remark}
For $\rho =1$ and for any $\theta$, the measure $\nu _{\rho ,\theta }^{n}$
reduces to the multivariate Gaussian measure (with independent components).
\end{remark}

We now prove that the moments of $\nu _{\rho ,\theta }^{1}$ are finite, so
that we can decompose the space $L^{2}\left( \mathbb{R},\nu _{\rho ,\theta
}^{1}\right) $ through the polynomials $H^{\rho,\theta}_n$. We obtain $%
H^{\rho,\theta}_n$ applying the Gram-Schmidt orthogonalization to the
monomials $x^{n}.$

\begin{lem}
\label{moments} Let $\rho \in (0,1]$ and $\theta >0$. The moments of $\nu
_{\rho ,\theta }^{1}$ are equal to zero, for $k=2n+1$, $n\in \mathbb{N}$ and
\begin{equation}
\int_{\mathbb{R}}x^{k}d\nu _{\rho ,\theta }^{1}(x)=\frac{(-1)^{n+1}2n!
\Gamma (\rho )\theta ^{\rho -n}}{n!2^n\Gamma (\rho ,\theta )}E_{1,\rho
+1-n}^{\rho }\left( -\theta \right) ,\qquad k=2n\text{, }n\in \mathbb{N}.
\label{yy2}
\end{equation}
The first polynomials $H_{n}^{\rho ,\theta },$ $n=0,1,2,3,$ orthogonal in $
L^{2}\left( \mathbb{R},\nu _{\rho ,\theta }^{1}\right) $ and with $\deg
H_{n}^{\rho ,\theta }=n,$ are given by
\begin{eqnarray}
H_{0}^{\rho ,\theta }(x)=1\qquad \qquad &&H_{1}^{\rho ,\theta }(x)=x
\label{pp} \\
H_{2}^{\rho ,\theta }(x) =x^{2}-\frac{\theta ^{\rho -1}e^{-\theta }}{\Gamma
(\rho ,\theta )}\qquad \qquad &&H_{3}^{\rho ,\theta }(x)=x^{3}-3x(1+(1-\rho
)\theta ^{-1}).  \notag
\end{eqnarray}
\end{lem}

\begin{proof}
We evaluate the derivatives of (\ref{yy3}): for $n=1$ we get
\begin{equation*}
\frac{d}{d\xi }\Phi _{\rho ,\theta }(\xi )=\frac{d}{d\xi }\frac{\Gamma (\rho
,\frac{1}{2}\xi ^{2}+\theta )}{\Gamma (\rho ,\theta )}=-\frac{\xi }{\Gamma
(\rho ,\theta )}e^{-(\frac{\xi ^{2}}{2}+\theta )}(\frac{\xi ^{2}}{2}+\theta
)^{\rho -1},
\end{equation*}
which vanishes, for $\xi =0.$ For $n > 1$ and $l=1,2,...$ we have instead
that
\begin{eqnarray*}
\frac{d^{l+1}}{d\xi ^{l+1}}\Phi _{\rho ,\theta }(\xi ) &=&-\frac{1}{\Gamma
(\rho ,\theta )}\frac{d^{l}}{d\xi ^{l}}\left[ \xi e^{-(\frac{\xi ^{2}}{2}%
+\theta )}(\frac{\xi ^{2}}{2}+\theta )^{\rho -1}\right] \\
&=&\frac{1}{\Gamma (\rho ,\theta )}\frac{d^{l}}{d\xi ^{l}}\left[ \xi
\sum_{j=0}^{\infty }(-1)^{j+1}\frac{(\frac{\xi ^{2}}{2}+\theta )^{j+\rho -1}%
}{j!} \right] \\
&=&\frac{1}{\Gamma (\rho ,\theta )}\frac{d^{l}}{d\xi ^{l}}\left[
\sum_{j=0}^{\infty }\frac{(-1)^{j+1}}{j!}\sum_{k=0}^{\infty }\frac{1}{2^k}%
\binom{j+\rho -1 }{k}\xi ^{2k+1}\theta ^{j+\rho -1-k}\right] \\
&=&\frac{1}{\Gamma (\rho ,\theta )}\sum_{j=0}^{\infty }\frac{(-1)^{j+1}}{j!}
\sum_{k=\left\lceil (l-1)/2\right\rceil }^{\infty }\frac{1}{2^k}\frac{(2k+1)!%
}{(2k+1-l)! }\binom{j+\rho -1}{k}\xi ^{2k+1-l}\theta ^{j+\rho -1-k}.
\end{eqnarray*}

For $\xi =0$ and for odd $l=2n+1$, the term $k=(l-1)/2=n$ is only one
different from zero, so that we get:
\begin{eqnarray*}
\left. \frac{d^{l+1}}{d\xi ^{l+1}}\Phi _{\rho ,\theta }(\xi
)\right\vert_{\xi =0}&=&-\frac{(2n+1)!\theta ^{\rho -n-1}}{n!2^n\Gamma (\rho
,\theta )}\sum_{j=0}^{\infty }\frac{(-\theta )^{j}}{j!}\frac{\Gamma (j+\rho )%
}{\Gamma(j+\rho -n)} \\
&=&-\frac{(2n+1)!\theta ^{\rho -n-1}\Gamma(\rho)}{n!2^n\Gamma (\rho ,\theta )%
} E_{1,\rho-n}^{\rho}(-\theta).  \label{pp3}
\end{eqnarray*}

Thus we obtain formula (\ref{yy2}) and the first even moments read
\begin{eqnarray}
\int_{\mathbb{R}}x^{2}d\nu _{\rho ,\theta }^{1}(x) &=&\frac{\theta ^{\rho
-1}e^{-\theta }}{\Gamma (\rho ,\theta )}  \label{pp2} \\
\int_{\mathbb{R}}x^{4}d\nu _{\rho ,\theta }^{1}(x), &=&\frac{3e^{-\theta }}{
\Gamma (\rho ,\theta )}\left[ \theta ^{\rho -1}+(1-\rho )\theta ^{\rho -2} %
\right] .  \notag
\end{eqnarray}

The polynomials in (\ref{pp}) follow from (\ref{pp2}), by solving the
following equations, for $H_{k}^{\rho ,\theta }(x)$, $k=0,1,...,$
\begin{equation*}
\mathbb{E}_{\nu _{\rho ,\theta }^{1}}\left[ \left(
a_{0}+a_{1}X+...+X^{r}\right) X^{k}\right] =0,
\end{equation*}
for $r=0,1,...k$.
\end{proof}

\begin{remark}
For $\rho =1$ and for any $\theta$, formula (\ref{yy2}) reduces to the $k$%
-th moment (for $k=2n$) of a Gaussian random variable with variance $1$:
\begin{eqnarray*}
\int_{\mathbb{R}}x^{2n}d\nu _{\rho ,\theta }^{1}(x) &=& \frac{(2n)!}{n!2^n}%
\frac{(-1)^{n+1}\theta^{1-n}}{e^{-\theta}}E^{1}_{1,2-n}(-\theta) \\
&=&\frac{(2n)!}{n!2^n}=n!!
\end{eqnarray*}
where we use the fact that $\theta ^{\rho -n}E_{1,\rho +1-n}^{\rho }\left(
-\theta \right) =\frac{d^{n-1}}{d\theta ^{n-1}}\left[ \theta ^{\rho
-1}E_{1,\rho }^{\rho }\left( -\theta \right) \right]$.\newline
Correspondingly, for $\rho =1$ and for any $\theta$, $H_{n}^{\rho ,\theta },$
given in (\ref{pp}), for $n=0,1,2,3$, reduce to the first four Hermite
polynomials.
\end{remark}

We can now extend the $n$-dimensional $\Gamma$-grey measure to the infinite
dimensional space $\mathcal{S}^{\prime }(\mathbb{R})$, dual of the space of
Schwartz functions $\mathcal{S}(\mathbb{R})$ (respectively $\mathcal{S}%
^{\prime}$ and $\mathcal{S}$, hereinafter).

Recalling that $\mathcal{S}\subset L^{2}(\mathbb{R},dx)\subset \mathcal{S}%
^{\prime }$ is a nuclear triple, we can define the measure $\nu _{\rho
,\theta }$ on $(\mathcal{S}^{\prime }, \mathcal{\sigma}^*)$ via the
Bochner-Minlos theorem, where $\mathcal{\sigma}^*$ is the $\sigma$-algebra
generated by the cylinders \cite{HID}.\newline
In analogy to the above definition of $\nu _{\rho ,\theta }^{n}$ in $\mathbb{%
R}^{n}$, we have the following:

\begin{definition}
For $\rho \in (0,1]$, $\theta>0$, the $\Gamma$-grey measure $\nu _{\rho
,\theta }$ is the unique probability measure fulfilling
\begin{equation}
\int_{\mathcal{S}^{\prime }}e^{i\left\langle x,\xi \right\rangle }d\nu
_{\rho ,\theta }(x)=\frac{\Gamma (\rho , \theta + \frac{1}{2}\left\langle
\xi ,\xi \right\rangle)}{\Gamma (\rho ,\theta )}\text{,\qquad }\xi \in
\mathcal{S}.
\end{equation}

We call $(\mathcal{S}^{\prime}, \sigma^*, \nu_{\rho,\theta})$ the $\Gamma$%
-grey noise space.
\end{definition}

\begin{remark}
For $\rho =1$ it reduces to the Gaussian white noise measure $\nu :=\nu
_{1,\theta }$, for any $\theta$.\newline
\end{remark}

The moments and the covariance of generalized stochastic processes on $(%
\mathcal{S}^{\prime},\sigma^*,\nu_{\rho,\theta})$ can be obtained by
considering those of the one-dimensional measure, given in Lemma \ref%
{moments}.

\begin{coro}
\label{cor1} Let $\rho \in (0,1]$ and $\theta >0$. Let $\xi ,\eta \in
\mathcal{S}$ and $n\in \mathbb{N}$, then $\int_{\mathcal{S}^{\prime
}}\left\langle x,\xi \right\rangle ^{2n+1}d\nu _{\rho ,\theta }(x)=0$ and
\begin{equation}
\int_{\mathcal{S}^{\prime }(\mathbb{R})}\left\langle x,\xi \right\rangle
^{2n}d\nu _{\rho ,\theta }(x)=\frac{(-1)^{n+1}(2n)!\Gamma (\rho )\theta
^{\rho -n}\left\langle \xi ,\xi \right\rangle ^{n}}{n!2^{n}\Gamma (\rho
,\theta )}E_{1,\rho +1-n}^{\rho }\left( -\theta \right) .  \label{ni}
\end{equation}

Moreover,%
\begin{equation}
\mathbb{E}_{\nu _{\rho ,\theta }}(\langle \omega ,\xi _{1}\rangle \langle \omega ,\xi _{2}\rangle )=%
\frac{\theta ^{\rho -1}e^{-\theta }}{\Gamma (\rho ,\theta )}\langle \xi
_{1},\xi _{2}\rangle ,  \label{ni2}
\end{equation}

for $\xi _{1},\xi _{2}\in \mathcal{S}$ and $\omega \in \mathcal{S}^{\prime }
$. Moreover, $\left\Vert \left\langle \cdot ,\xi \right\rangle \right\Vert
_{L^{2}(\nu _{\rho ,\theta })}^{2}=\theta ^{\rho -1}e^{-\theta }\Vert \xi
\Vert ^{2}/\Gamma (\rho ,\theta ).$
\end{coro}

\begin{proof}
Since the moments are easily obtained from Lemma \ref{moments}, we just
compute the covariance as

\[ \mathbb{E}(\langle \omega ,\xi _{1}\rangle
\langle \omega ,\xi _{2}\rangle )=i^{-2}D_{a_{1},a_{2}}\left. \Big(\frac{%
	\Gamma (\rho ,\theta +\frac{1}{2}\Vert a_{1}\xi _{1}+a_{2}\xi _{2}\Vert ^{2})%
}{\Gamma (\rho ,\theta )}\Big)\right\vert _{_{a_{1}=a_{2}=0}} \text{ for } \omega
\in \mathcal{S}^{\prime },\]
where $D_{a_1,a_2}$ is the derivative w.r.t. $a_1$ and $a_2$.

We can write $\Vert a_{1}\xi _{1}+a_{2}\xi _{2}\Vert ^{2}=\langle
a_{1}\xi _{1}+a_{2}\xi _{2},a_{1}\xi _{1}+a_{2}\xi _{2}\rangle $
and thanks to the bi-linearity we have $a_{1}^{2}\Vert \xi
_{1}\Vert ^{2}+2a_{1}a_{2}\langle
\xi _{1},\xi _{2}\rangle +a_{2}^{2}\Vert \xi _{2}\Vert ^{2}=:F(a_{1},a_{2})$%
. Hence,
\begin{eqnarray*}
&&D_{a_{1},a_{2}}\Big(\frac{\Gamma (\rho ,\theta +\frac{1}{2}\Vert a_{1}\xi
_{1}+a_{2}\xi _{2}\Vert ^{2})}{\Gamma (\rho ,\theta )}\Big)=\frac{1}{\Gamma
(\rho ,\theta )}D_{a_{1},a_{2}}\Big(\Gamma (\rho ,\theta +\frac{1}{2}%
F(a_{1},a_{2}))\Big) \\
&=&-\frac{1}{\Gamma (\rho ,\theta )}D_{a_{1}}\Big((\theta +\frac{1}{2}%
F(a_{1},a_{2}))^{\rho -1}e^{-(\theta +\frac{1}{2}F(a_{1},a_{2}))}(a_{1}%
\langle \xi _{1},\xi _{2}\rangle +a_{2}\Vert \xi _{2}\Vert ^{2})\Big) \\
&=&-\frac{1}{\Gamma (\rho ,\theta )}\Big((\rho -1)(\theta +\frac{1}{2}%
F(a_{1},a_{2}))^{\rho -2}e^{-(\theta +\frac{1}{2}F(a_{1},a_{2}))}(a_{1}\Vert
\xi _{1}\Vert ^{2}+a_{2}\langle \xi _{1},\xi _{2}\rangle )(a_{1}\langle \xi
_{1},\xi _{2}\rangle +a_{2}\Vert \xi _{2}\Vert ^{2}) \\
&\quad &+(\theta +\frac{1}{2}F(a_{1},a_{2}))^{\rho -1}e^{-(\theta +\frac{1}{2%
}F(a_{1},a_{2}))}(-(a_{1}\Vert \xi _{1}\Vert ^{2}+a_{2}\langle \xi _{1},\xi
_{2}\rangle ))(a_{1}\langle \xi _{1},\xi _{2}\rangle +a_{2}\Vert \xi
_{2}\Vert ^{2}) \\
&\quad &+(\theta +\frac{1}{2}F(a_{1},a_{2}))^{\rho -1}e^{-(\theta +\frac{1}{2%
}F(a_{1},a_{2}))}\langle \xi _{1},\xi _{2}\rangle \Big),
\end{eqnarray*}%
which, taking $a_{1}=a_{2}=0$ and multiplying by $i^{-2}=-1,$ coincides with
(\ref{ni2}).
\end{proof}

We now want to prove that the $\Gamma$-grey measure
$\nu_{\rho,\theta}$ belongs to the class of measures for which the
Appell systems exist. The latter are bi-orthogonal polynomials
which replace the Wick-ordered polynomials of Gaussian analysis
and have been proved to be fundamental tools in the non-Gaussian
context. To this aim, it is sufficient to prove the following
conditions are satisfied (see \cite{KSWY98} for details):

\begin{enumerate}[label=\textbf{C\arabic*}]

\item For $\rho \in (0,1]$ and $\theta > 0$, $\nu_{\rho,\theta}$ has an
analytic Laplace transform in a neighborhood of zero, i.e. the following
mapping is holomorphic in a neighborhood $\mathcal{U}\subset \mathcal{S}_{%
\mathbb{C}}$ of zero:
\begin{equation*}
\mathcal{S}_{\mathbb{C}} \ni \phi \mapsto \ell_{\nu} (\phi):=\int_{\mathcal{S%
}^{\prime}}\exp{\langle x,\phi \rangle }d\nu_{\rho,\theta}(x) \in \mathbb{C}
\end{equation*}%
\newline

\item For $\rho \in (0,1]$ and $\theta > 0$, $\nu_{\rho,\theta}(\mathcal{U}%
)>0$ for any non-empty open subset $\mathcal{U}\subset \mathcal{S}^{\prime}$.%
\newline
\end{enumerate}

As far as C1 is concerned, we recall in Appendix B some definitions and
well-known results on holomorphic property; on the basis of the latter, we
show that for $\rho \in (0,1]$ and $\theta>0$ the measure $\nu_{\rho,\theta}$
admits a Laplace transform defined only on a subset of $\mathcal{S}_{\mathbb{%
C}}$ but it is holomorphic on that subset and it is positive on non-empty,
open subsets.\newline

First we show that $\ell_{\nu}(\xi)$ is well-defined on a subset of $\mathcal{S}$.

\begin{lem}
\label{Lap in 1} Let $\rho \in (0,1)$, $\theta >0$ and $\lambda \in \mathbb{R%
}/\{0\}$, then the exponential function $\mathcal{S}^{\prime }\ni \omega
\mapsto e^{|\lambda \langle x,\phi \rangle |}$ is integrable and
\begin{equation}
\int_{\mathcal{S}^{\prime }}e^{\lambda \langle x,\phi \rangle }d\nu _{\rho
,\theta }(x)=\frac{\Gamma (\rho ,\theta -\frac{\lambda ^{2}}{2}\langle \phi
,\phi \rangle )}{\Gamma (\rho ,\theta )},\qquad \text{ for }\phi \in B_{%
\sqrt{2\theta /\lambda ^{2}}}(0).  \label{le}
\end{equation}
\end{lem}

\begin{proof}
For $\lambda \in \mathbb{R}/\{0\}$ we start by proving the integrability. We
can define the monotone increasing sequence $g_{N}(\cdot ):=\sum_{n=0}^{N}%
\frac{1}{n!}|\langle \cdot ,\lambda \phi \rangle |^{n}$. We divide the
elements of $g_{N}$ into odd and even terms,
\begin{equation*}
g_{N}(\cdot )=\sum_{n=0}^{\lfloor N/2\rfloor }\frac{1}{(2n)!}|\langle \cdot
,\lambda \phi \rangle |^{2n}+\sum_{n=0}^{\lceil N/2\rceil -1}\frac{1}{(2n+1)!%
}|\langle \cdot ,\lambda \phi \rangle |^{2n+1}
\end{equation*}%
and we apply the integral to each term. For the even terms we get:
\begin{equation*}
\frac{1}{(2n)!}\int_{\mathcal{S}^{\prime }}|\langle x,\lambda \phi \rangle
|^{2n}d\nu _{\rho, \theta }(x)=\frac{\Gamma (\rho )}{\Gamma (\rho ,\theta )}%
\frac{(-1)^{n+1}\theta ^{\rho -n}E_{1,\rho +1-n}^{\rho }\left( -\theta
\right) }{n!2^{n}}\langle \lambda \phi ,\lambda \phi \rangle ^{n}.
\end{equation*}%
\newline
By considering that $\theta ^{\rho -n}E_{1,\rho +1-n}^{\rho }(-\theta )=%
\frac{d^{n-1}}{d\theta ^{n-1}}[\theta ^{\rho -1}E_{1,\rho }^{\rho }(-\theta
)]$ (see (1.9.6) in \cite{KIL}) and $E_{1,\rho }^{\rho }(-\theta )=\frac{1}{%
\Gamma (\rho )}e^{-\theta }$, we get
\begin{eqnarray*}
\theta ^{\rho -n}E_{1,\rho +1-n}^{\rho }(-\theta ) &=&\frac{1}{\Gamma (\rho )%
}\frac{d^{n-1}}{d\theta ^{n-1}}[\theta ^{\rho -1}e^{-\theta }]=\frac{1}{%
\Gamma (\rho )}\frac{\partial ^{n-1}}{\partial \theta ^{n-1}}\left[ -\frac{%
\partial }{\partial \theta }\Gamma (\rho ,\theta )\right]  \\
&=&-\frac{1}{\Gamma (\rho )}\frac{\partial ^{n}}{\partial \theta ^{n}}\Gamma
(\rho ,\theta )=-\frac{1}{\Gamma (\rho )}\left. \frac{\partial ^{n}}{%
\partial x^{n}}\Gamma (\rho ,\theta +x)\right\vert _{x=0}
\end{eqnarray*}

Hence, we have that
\begin{equation*}
(-1)^{n+1}\theta ^{\rho -n}E_{1,\rho +1-n}^{\rho }\left( -\theta \right) =%
\frac{(-1)^{n}}{\Gamma (\rho )}\left. \frac{\partial ^{n}}{\partial x^{n}}%
\Gamma (\rho ,\theta +x)\right\vert _{x=0},
\end{equation*}%
so that each even term is equal to:
\begin{equation*}
\frac{1}{(2n)!}\int_{\mathcal{S}^{\prime }}|\langle x,\lambda \phi \rangle
|^{2n}d\nu _{\rho, \theta }(x)=\frac{1}{\Gamma (\rho ,\theta )}\frac{\left.
\frac{\partial ^{n}}{\partial x^{n}}\Gamma (\rho ,\theta +x)\right\vert
_{x=0}}{n!2^{n}}(-\langle \lambda \phi ,\lambda \phi \rangle )^{n}=:E(n).
\end{equation*}%
We can estimate the odd terms using the Cauchy-Schwarz inequality on $L^{2}(%
\mathcal{S}^{\prime },\sigma ^{\ast },\nu _{\rho, \theta })$ and the
inequality $st\leq 1/2(s^{2}+t^{2})$, for $s,t\in \mathbb{R}$:

\begin{eqnarray*}
&&\frac{1}{(2n+1)!}\int_{\mathcal{S}^{\prime }}|\langle x,\lambda \phi
\rangle |^{2n+1}d\nu _{\rho, \theta }(x) \\
&=&\frac{1}{(2n+1)!}\int_{\mathcal{S}^{\prime }}|\langle x,\lambda \phi
\rangle |^{n+1}|\langle x,\lambda \phi \rangle |^{n}d\nu _{\rho, \theta }(x)
\\
&\leq &\frac{1}{(2n+1)!}\Big(\int_{\mathcal{S}^{\prime }}|\langle x,\lambda
\phi \rangle |^{2n+2}d\nu _{\rho, \theta }(x)\Big)^{1/2}\Big(\int_{\mathcal{S%
}^{\prime }}|\langle x,\lambda \phi \rangle |^{2n}d\nu _{\rho, \theta }(x)%
\Big)^{1/2} \\
&\leq &\frac{1}{(2n+1)!}\Big(\frac{1}{2}\int_{\mathcal{S}^{\prime }}|\langle
x,\lambda \phi \rangle |^{2n+2}d\nu _{\rho, \theta }(x)+\frac{1}{2}\int_{%
\mathcal{S}^{\prime }}|\langle x,\lambda \phi \rangle |^{2n}d\nu _{\rho
,\theta }(x)\Big) \\
&\leq &\frac{1}{(2n+1)!}\Big(\frac{(2n+2)!}{(n+1)!2^{n+2}\Gamma (\rho
,\theta )}\left. \frac{\partial ^{n+1}}{\partial x^{n+1}}\Gamma (\rho
,\theta +x)\right\vert _{_{x=0}}(-\langle \lambda \phi ,\lambda \phi \rangle
)^{n+1}\Big) \\
&\quad &+\frac{1}{(2n+1)!}\Big(\frac{(2n)!}{n!2^{n+1}\Gamma (\rho ,\theta )}%
\left. \frac{\partial ^{n}}{\partial x^{n}}\Gamma (\rho ,\theta
+x)\right\vert _{_{x=0}}(-\langle \lambda \phi ,\lambda \phi \rangle )^{n}%
\Big) \\
&=&\frac{1}{\Gamma (\rho ,\theta )}\frac{1}{n!2^{n+1}}\left. \frac{\partial
^{n+1}}{\partial x^{n+1}}\Gamma (\rho ,\theta +x)\right\vert
_{_{x=0}}(-\langle \lambda \phi ,\lambda \phi \rangle )^{n+1} \\
&\quad &+\frac{1}{(2n+1)}\frac{1}{\Gamma (\rho ,\theta )}\frac{1}{n!2^{n+1}}%
\left. \frac{\partial ^{n}}{\partial x^{n}}\Gamma (\rho ,\theta
+x)\right\vert _{_{x=0}}(-\langle \lambda \phi ,\lambda \phi \rangle )^{n} \\
&=:&O^{\prime }(n)+O^{\prime \prime }(n).
\end{eqnarray*}

Thus, by integrating $g_N$, we get that

\begin{equation*}
\int_{\mathcal{S}^{\prime }}g_{N}(x)d\nu _{\rho, \theta }(x)\leq
\sum_{n=0}^{\lfloor N/2\rfloor }E(n)+\sum_{n=0}^{\lceil N/2\rceil
-1}O^{\prime }(n)+\sum_{n=0}^{\lceil N/2\rceil -1}O^{\prime \prime }(n).
\end{equation*}%
We have that the sum of the even terms $E(n)$ converges to $\frac{1}{\Gamma
(\rho ,\theta )}\Gamma (\rho ,\theta -\lambda ^{2}/2\langle \phi ,\phi
\rangle )$ if $\langle \phi ,\phi \rangle =\Vert \phi \Vert ^{2}<2\theta
/\lambda ^{2}$, as the Taylor expansion for $\Gamma (\rho ,\eta )$ holds for
$\rho \in (0,1),$ if $\eta >0$. For the odd terms $O^{\prime }(n)$ we have
that:

\begin{eqnarray*}
\sum_{n=0}^{\lceil N/2\rceil -1}O^{\prime }(n) &=&\frac{1}{\Gamma (\rho
,\theta )}\sum_{n=0}^{\lceil N/2\rceil -1}\frac{n+1}{2^{n+1}}\frac{\left.
\frac{\partial ^{n+1}}{\partial x^{n+1}}\Gamma (\rho ,\theta +x)\right\vert
_{x=0}}{(n+1)!}(-\langle \lambda \phi ,\lambda \phi \rangle )^{n+1} \\
&\leq &\frac{1}{2\Gamma (\rho ,\theta )}\sum_{m=1}^{\lceil N/2\rceil }\frac{%
\left. \frac{\partial ^{m}}{\partial x^{m}}\Gamma (\rho ,\theta
+x)\right\vert _{x=0}}{m!}(-\langle \lambda \phi ,\lambda \phi \rangle )^{m},
\end{eqnarray*}%
where the last sum converges. On the other hand the sum of the odd terms $%
O^{\prime \prime }(n)$ converges since

\begin{eqnarray*}
\sum_{n=0}^{\lceil N/2\rceil -1}O^{\prime \prime }(n) &=&\sum_{n=0}^{\lceil
N/2\rceil -1}\frac{1}{(2n+1)2^{n+1}}\frac{1}{\Gamma (\rho ,\theta )}\frac{%
\left. \frac{\partial ^{n}}{\partial x^{n}}\Gamma (\rho ,\theta
+x)\right\vert _{x=0}}{n!}(-\langle \lambda \phi ,\lambda \phi \rangle )^{n}
\\
&<&\frac{1}{\Gamma (\rho ,\theta )}\sum_{n=0}^{\lceil N/2\rceil -1}\frac{%
\left. \frac{\partial ^{n}}{\partial x^{n}}\Gamma (\rho ,\theta
+x)\right\vert _{x=0}}{n!}(-\langle \lambda \phi ,\lambda \phi \rangle )^{n}%
\underset{N\rightarrow \infty }{\longrightarrow }\frac{1}{\Gamma (\rho
,\theta )}\Gamma (\rho ,\theta -\lambda ^{2}\langle \phi ,\phi \rangle ).
\end{eqnarray*}

Therefore, by applying the monotone convergence theorem (as each term is
positive), we get, for $\phi \in B_{\sqrt{2\theta/\lambda^2}}(0)$, that:

\begin{eqnarray*}
\int_{\mathcal{S}^{\prime }}e^{\langle x,\lambda \phi \rangle }d\nu _{\rho
,\theta }(x) &=&\lim_{N\rightarrow \infty }\int_{\mathcal{S}^{\prime
}}g_{N}(x)d\nu _{\rho ,\theta }(x) \\
&=&\frac{1}{\Gamma (\rho ,\theta )}\sum_{n\geq 0}\frac{\left. \frac{\partial
^{n}}{\partial x^{n}}\Gamma (\rho ,\theta +x)\right\vert _{x=0}}{n!}%
(-1/2)^{n}\langle \lambda \phi ,\lambda \phi \rangle ^{n},
\end{eqnarray*}%
which coincides with (\ref{le}).
\end{proof}

Now, we prove that $\ell_{\nu}(\xi)$ is holomorphic on some neighborhood of $%
0$ in $\mathcal{S}_{\mathbb{C}}$ for $\rho \in (0,1)$ and $\theta>0$. Hence
we have that $\ell_{\nu}(\xi)$ is holomorphic on $\mathcal{U}_\theta:=B_{%
\sqrt{2\theta}}(0)\oplus i \mathcal{S}=\{\xi_1+i\xi_2| \xi_1
\in B_{\sqrt{2\theta}}(0) \text{ and } \xi_2 \in \mathcal{S}\}$.

\begin{theorem}
\label{ThmLapHolo} Let $\rho \in (0,1)$ and $\theta >0$, then the function
\begin{equation*}
\mathcal{S}_{\mathbb{C}}\supset \mathcal{U}_{\theta }\ni \xi \mapsto \int_{%
\mathcal{S}^{\prime }}e^{\langle x,\xi \rangle }d\nu _{\rho ,\theta }(x)
\end{equation*}%
is holomorphic from $\mathcal{U}_{\theta }$ to $\mathbb{C}$.
\end{theorem}

\begin{proof}
We show that it is bounded on $\mathcal{U}_{\theta }$. Let $\xi \in \mathcal{%
U}_{\theta }$, then we have that

\begin{equation*}
|\ell _{\nu }(\xi )|\leq \int_{\mathcal{S}^{\prime }}|e^{\langle x,\xi
\rangle }|d\nu _{\rho ,\theta }(x)
\end{equation*}%
Noting that $|e^{\langle x,\xi \rangle }|=|e^{\langle x,\xi _{1}\rangle
}||e^{-i\langle x,\xi _{2}\rangle }|=|e^{\langle x,\xi _{1}\rangle
}|=e^{\langle x,\xi _{1}\rangle }$, for $\xi =\xi _{1}+i\xi _{2},$ we get
\begin{equation*}
\int_{\mathcal{S}^{\prime }}|e^{\langle x,\xi \rangle }|d\nu _{\rho ,\theta
}(x)=\int_{\mathcal{S}^{\prime }}e^{\langle x,\xi _{1}\rangle }d\nu _{\rho
,\theta }(x)=\frac{\Gamma (\rho ,\theta -1/2\Vert \xi _{1}\Vert ^{2})}{%
\Gamma (\rho ,\theta )}<\infty ,
\end{equation*}%
by using Lemma \ref{Lap in 1}, for the second equality.\newline Now we show that, for $%
\xi =\xi _{1}+i\xi _{2}\in \mathcal{U}_{\theta }$, $\eta =\eta
_{1}+i\eta
_{2}\in \mathcal{S}_{\mathbb{C}}$ and $z\in B_{r}(0)$ where $0<r<\frac{\sqrt{%
2\theta }-\Vert \xi _{1}\Vert }{3(\Vert \eta _{1}\Vert +\Vert \eta _{2}\Vert
)}$, the function $\mathbb{C}\supset B_{r}(0)\ni z\mapsto \ell _{\nu }(\xi
+z\eta )=:f(z)\in \mathbb{C}$ is continous. The radius length is such that,
for all $\lambda \in B_{r}(0)$, we have $\xi +\lambda \eta \in \mathcal{U}%
_{\theta }$. \newline
We take $\{z_{n}\}_{n\in \mathbb{N}}\subset B_{r}(0)$ such that $%
z_{n}\rightarrow z$, for $n\rightarrow \infty $. Denoting by
$(\cdot )_{1}$ the real part of a function in $\mathcal{S}_\C$, we have that

\begin{eqnarray*}
|f(z)-f(z_n)|&\leq& \int_{\mathcal{S}^{\prime}}|e^{\langle x, \xi + z \eta
\rangle}-e^{\langle x, \xi + z_n \eta \rangle} |d\nu_{\rho,\theta}(x) \\
&\leq& \int_{\mathcal{S}^{\prime}}|e^{\langle x, \xi\rangle}||e^{\langle x
,z \eta \rangle}-e^{\langle x, z_n \eta\rangle}|d\nu_{\rho,\theta}(x) \\
&\leq& \int_{\mathcal{S}^{\prime}}e^{\langle x, \xi_1\rangle + \langle x,
(z_n\eta)_1 + ((z-z_n)\eta)_1 \rangle}d\nu_{\rho,\theta}(x)
\end{eqnarray*}
We note that $|e^{\langle x ,z \eta \rangle}-e^{\langle x, z_n
\eta\rangle}|=|e^{ \langle x,z_n \eta\rangle}||e^{\langle x,
(z-z_n)\eta \rangle}-1|$. Moreover, for sufficiently large $n$, we
have that $|e^{\langle x, (z-z_n)\eta \rangle}-1|\leq e^{\langle
x, ((z-z_n)\eta)_1 \rangle}$. Since $|z_n| +
|z-z_n|<3r$ for each $n$, we can ensure that $\xi + 3r\eta \in \mathcal{U}%
_{\theta}$, so that $e^{\langle x, (\xi + 3r\eta)_1 \rangle}\in L^1(\mathcal{S%
}^{\prime},\sigma^*,\nu_{\rho,\theta})$. Hence, we can apply the
dominated convergence theorem to gain the continuity of $f$ in $z
\in B_{r}(0)$, as follows

\begin{eqnarray*}
\lim_{n \to \infty}|f(z)-f(z_n)|&\leq& \lim_{n \to \infty} \int_{\mathcal{S}%
^{\prime}}|e^{\langle x, \xi_1\rangle}||e^{\langle x ,z \xi_2
\rangle}-e^{\langle x, z_n \xi_2\rangle}|d\nu_{\rho,\theta}(x) \\
&=&\int_{\mathcal{S}^{\prime}}\lim_{n \to \infty}|e^{\langle x,
\xi_1\rangle}||e^{z \langle x , \xi_2 \rangle}-e^{z_n \langle x,
\xi_2\rangle}|d\nu_{\rho,\theta}(x)=0.
\end{eqnarray*}
Now we apply the Morera's theorem to show that $f(z)$ is
holomorphic, which means that $\ell_{\nu}(\xi)$ is G-holomorphic
on $\mathcal{U}_{\theta}$ (see Definition B.3 in Appendix). Let
$\gamma$ be a closed and bounded curve in $B_{r}(0)\subset
\mathbb{C}$, since $\gamma$ is compact and
$\int_{\mathcal{S}^{\prime}}e^{\langle x, \xi + z\eta
\rangle}d\nu_{\rho,\theta}(x)<\infty$, we can use the Fubini
theorem to get:
\begin{equation*}
\int_{\gamma}\int_{\mathcal{S}^{\prime}}e^{\langle x, \xi + z\eta
\rangle}d\nu_{\rho,\theta}(x)dz=\int_{\mathcal{S}^{\prime}}\int_{\gamma}e^{%
\langle x, \xi + z\eta \rangle}dz d\nu_{\rho,\theta}(x)=0
\end{equation*}
as the exponential function is holomorphic. By the Morera's theorem and by
Lemma \ref{Gholobound} in Appendix B, we have that $\ell_{\nu}$ is
holomorphic on $\mathcal{U}_\theta$.
\end{proof}

\begin{remark}
It is easy to check that $B_{r_1}(0)\oplus i B_{r_2}(0)$ is an open sen in
the topology induced by $\langle \cdot, \cdot \rangle_{\mathcal{H}_\C}$, as
follows: let us define the projections of an element of $\mathcal{H}_{%
\mathbb{C}}$ as
\begin{equation*}
\pi_1:\mathcal{H}_{\mathbb{C}} \to \mathcal{H}: \xi_1+i\xi_2 \mapsto \xi_1
\end{equation*}
and
\begin{equation*}
\pi_2:\mathcal{H}_{\mathbb{C}} \to \mathcal{H}: \xi_1+i\xi_2 \mapsto \xi_2.
\end{equation*}
Let $x \in B_{r_1}(0)\oplus i B_{r_2}(0)$, then we have $\pi_1(x)\in
B_{r_1}(0) $ and $\pi_2(x)\in B_{r_2}(0)$. The sets $B_{r_1}(0)$ and $%
B_{r_2}(0)$ are open in the topology of $\mathcal{H}$; then $\exists
\epsilon_1, \epsilon_2 >0$ such that $B_{\epsilon_1}(\pi_1(x)) \subset
B_{r_1}(0)$ and $B_{\epsilon_2}(\pi_2(x)) \subset B_{r_2}(0)$. Let $\epsilon
= \min\{\epsilon_1, \epsilon_2\}$, then we have that $B^{\mathbb{C}%
}_{\epsilon}(x) \subset B_{r_1}(0)\oplus i B_{r_2}(0)$.
\end{remark}

In order to verify that C2 is satisfied by
$\nu_{\rho,\theta}$, we prove that, for $\rho \in (0,1)$ and
$\theta >0$, they are always strictly positive on non-empty, open
subsets, by resorting to their representation as mixture of
Gaussian measures.

\begin{theorem}
For any open, non-empty set $U\subset \mathcal{S}^{\prime}$ and
for any $\rho \in (0,1)$, $\theta >0$,
we have that $%
\nu_{\rho,\theta}(U)>0$.
\end{theorem}

\begin{proof}
By applying Theorem 4.5 in \cite{GRO}, it is sufficient to prove that $%
\nu_{\rho,\theta}$ is an elliptically contoured measure, i.e. if we denote
by $\mu^{s}$ the centered Gaussian measure on $\mathcal{S}^{\prime}$ with
variance $s>0$, the following holds:
\begin{equation}  \label{Gaussianity}
\nu_{\rho,\theta}= \int_{0}^{\infty} \mu^{s} d \mu_{\rho,\theta}(s),
\end{equation}
where $\mu_{\rho,\theta}$ is the measure defined on $(0,\infty)$ by (\ref%
{1dim}). The identity in (\ref{Gaussianity}) can be checked by considering
that

\begin{equation*}
\int_{\mathcal{S}^{\prime}}e^{i \langle x, \xi \rangle} d\mu^s(x)=\exp\big\{ %
-\frac{s}{2}\langle \xi, \xi \rangle\big\}, \quad \xi \in \mathcal{S}
\end{equation*}
and thus, by (\ref{1dim}),

\begin{equation}
\int_{0}^{\infty} \exp\big\{ -\frac{s}{2} \langle \xi, \xi \rangle\big\} d
\mu_{\rho,\theta}(s)=\frac{\Gamma(\rho, \theta + \frac{1}{2} \langle \xi,
\xi \rangle)}{\Gamma(\rho,\theta)},  \label{ec}
\end{equation}
which coincides with $\int_{\mathcal{S}^{\prime}} e^{i \langle x, \xi
\rangle}d\nu_{\rho,\theta}(x)$.
\end{proof}

\begin{remark}
	For $\rho=1$, C1 and C2 are satisfied because $\nu_{1,\theta}$ is Gaussian, for each $\theta\geq 0$.
\end{remark}
\section{The tempered Gamma-grey Brownian motion as generalized
stochastic process}

We can now consider the fractional operator $M_{-}^{\alpha /2}$ defined, for
any $f\in \mathcal{S}$, as
\begin{equation*}
M_{-}^{\alpha /2}f:=\left\{
\begin{array}{l}
\sqrt{C_{\alpha }}D_{\_}^{(1-\alpha )/2}f,\qquad \alpha \in (0,1) \\
f,\qquad \qquad \qquad \qquad \alpha =1 \\
\sqrt{C_\alpha}I_{\_}^{(\alpha -1)/2}f, \qquad \alpha \in (1,2)%
\end{array}
\right. ,
\end{equation*}
where
\begin{equation*}
D_{\_}^{\beta }\,f(x):=-\frac{1}{\Gamma (1-\beta )}\frac{d}{dx}
\int_{x}^{\infty }f(t)(t-x)^{-\beta }dt,\quad x\in \mathbb{R},\;\beta \in
(0,1),
\end{equation*}
is the Riemann-Liouville fractional derivative and

\begin{equation*}
I_{\_}^{\beta }\,f(x):=\frac{1}{\Gamma (\beta )} \int_{x}^{\infty
}f(t)(t-x)^{\beta-1 }dt,\quad x\in \mathbb{R},\;\beta \in (0,1),
\end{equation*}
is the Riemann-Liouville fractional integral.

We extend the dual pairing $\left\langle \cdot ,\cdot \right\rangle $ to $%
\mathcal{S}^{\prime }(\mathbb{R})\times L^{2}\left( \mathbb{R},dx\right) $
and by considering that $M_{-}^{\alpha /2}1_{[0,t)}\in L^{2}\left( \mathbb{R}%
,dx\right)$, where $1_{[a,b)}$ is the indicator function of $[a,b)$, we
introduce the tempered $\Gamma$-grey Brownian motion (hereafter $\Gamma$%
-GBM) as follows:

\begin{definition}
Let $\alpha \in (0,2)$, $\rho \in (0,1]$ and $\theta>0$. The tempered $%
\Gamma $-GBM is defined on the probability space $\left( \mathcal{S}^{\prime
}(\mathbb{R}),\sigma^*,\nu _{\rho ,\theta }\right) $ as the generalized
process
\begin{equation}
B_{\alpha ,\rho }^{\theta }(t,\omega ):=\left\langle \omega ,M_{-}^{\alpha
/2}1_{[0,t)}\right\rangle ,\qquad t\geq 0,\;\omega \in \mathcal{S}^{\prime
}( \mathbb{R}).  \label{i4}
\end{equation}
\end{definition}

We notice that for each $t$, $B_{\alpha ,\rho }^{\theta }(t,\cdot) \in
L^{2}(\nu_{\rho,\theta})$.

\begin{remark}
For $\rho=1$, $\alpha=1$ and for each $\theta$, we have that $%
B^{\theta}_{1,1}$ is a Brownian Motion, indeed for each $t$, $%
B_t(\omega)=\langle \omega, 1_{[0,t)}\rangle \in L^2(\nu)$ where $\nu$ is
Gaussian.
\end{remark}

In order to study the continuity of this process, we recall the following
relationship obtained in \cite{GRO2}:
\begin{equation}
\left\langle M_{-}^{\alpha /2}\xi ,M_{-}^{\alpha /2}\eta \right\rangle
_{L^{2}(\mathbb{R},dx)}=C_{\alpha }\int_{\mathbb{R}}|x|^{1-\alpha }\widehat{%
\xi }(x)\widehat{\eta }(x)dx,\qquad \xi ,\eta \in \mathcal{S}(\mathbb{R}),
\label{gro}
\end{equation}%
where $\widehat{f}(x):=\frac{1}{\sqrt{2\pi }}\int_{\mathbb{R}%
}e^{i\left\langle \omega ,x\right\rangle }f(\omega )d\omega $ denotes the
Fourier transform of $f(\cdot )$ (see also \cite{MIS}, for details).

It is proved in \cite{GRO2} that (\ref{gro}) holds not only on $\mathcal{S}$%
, but also for indicator functions and that%
\begin{equation*}
\left\langle M_{-}^{\alpha /2}1_{[0,t)},M_{-}^{\alpha
/2}1_{[0,s)}\right\rangle _{L^{2}(\mathbb{R},dx)}=\frac{1}{2}(s^{\alpha
}+t^{\alpha }-|t-s|^{\alpha }).
\end{equation*}

Similarly to what was done in \cite{GRO2} for the ggBm, it is easy
to prove the following result.

\begin{theorem}
For $\alpha \in (0,2),$ $\rho \in (0,1]$ and $\theta >0,$ the tempered $%
\Gamma$-GBM has a $\gamma$-H\"{o}lder continuous version with
$\gamma < \alpha /2$.
\end{theorem}

\begin{proof}
In order to apply the Kolmogorov's continuity theorem, we only need to show
that

\begin{equation}
\mathbb{E}_{\nu _{\rho ,\theta }}\left( \left( B_{\alpha ,\rho }^{\theta
}(t)-B_{\alpha ,\rho }^{\theta }(s)\right) ^{2n}\right) \leq K|t-s|^{q+1},
\label{kk}
\end{equation}
for some $q>0$ and $s,t\geq 0.$ By definition and by recalling (\ref{ni}),
we have that, for $s<t,$
\begin{eqnarray*}
&&\mathbb{E}_{\nu _{\rho ,\theta }}\left( \left\vert B_{\alpha ,\rho
}^{\theta }(t)-B_{\alpha ,\rho }^{\theta }(s)\right\vert ^{2n}\right) \\
&=&\int_{\mathcal{S}^{\prime }(\mathbb{R})}\left\langle \omega
,M_{-}^{\alpha /2}1_{[s,t)}\right\rangle ^{2n}d\nu _{\rho ,\theta }(\omega )
\\
&=&\frac{(-1)^{n+1}(2n)!\Gamma (\rho )\theta ^{\rho -n}}{n!2^n\Gamma (\rho
,\theta )}E_{1,\rho +1-n}^{\rho }\left( -\theta \right) \left\langle
M_{-}^{\alpha /2}1_{[s,t)},M_{-}^{\alpha /2}1_{[s,t)}\right\rangle ^{n} \\
.
\end{eqnarray*}
We now prove that
\begin{equation*}
K_{\theta , \rho}^{n}:=\frac{(-1)^{n+1}(2n)!\Gamma (\rho )\theta ^{\rho -n}}{%
n!2^n\Gamma (\rho ,\theta )}E_{1,\rho +1-n}^{\rho }\left( -\theta \right)
\end{equation*}
is positive, for any $n$, by considering (\ref{di}), as follows
\begin{eqnarray*}
(-1)^{n+1}\theta ^{\rho -n}E_{1,\rho +1-n}^{\rho }\left( -\theta \right)
&=&(-1)^{n+1}\frac{d^{n-1}}{d\theta ^{n-1}}\left[ \theta ^{\rho -1}E_{1,\rho
}^{\rho }\left( -\theta \right) \right] \\
&=&(-1)^{n+1}\sum_{j=0}^{n-1}\binom{n-1}{j}\frac{d^{j}}{d\theta ^{j}}\left[
\theta ^{\rho -1}\right] \frac{d^{n-1-j}}{d\theta ^{n-1-j}}\frac{e^{-\theta }%
}{\Gamma (\rho)} \\
&=&\frac{1}{\Gamma (\rho)}\sum_{j=0}^{n-1}\binom{n-1}{j}\left\{ (-1)^{j}%
\frac{d^{j}}{d\theta ^{j}} \left[ \theta ^{\rho -1}\right] \right\} \left\{
(-1)^{n-1-j}\frac{d^{n-1-j} }{d\theta ^{n-1-j}}e^{-\theta }\right\} \geq 0.
\end{eqnarray*}
In the last step we resorted to the complete monotonicity of both the
Prabhakar function (in the special case $\beta =\gamma =\rho <1$) and of $
\theta ^{\rho -1},$ for $\rho <1.$

We now apply Proposition 3.8 of \cite{GRO2}, which shows that
$\left\langle M_{-}^{\alpha /2}1_{[s,t)},M_{-}^{\alpha
/2}1_{[s,t)}\right\rangle ^{n}=(t-s)^{\alpha n}$, so that
(\ref{kk}) holds for $q=\alpha n-1>0$. The case $s>t$ can be
treated analogously, so that the sufficient condition of the
Kolmogorov's continuity theorem is satisfied and the
H\"{o}lder-continuity parameter is $\gamma < \frac{q+1}{2n} =
\frac{\alpha}{2}$.
\end{proof}

\begin{remark}
The previous result agrees with the well-known $\gamma$-H\"{o}lder
continuity of the fractional Brownian motion with $\gamma <H$.
\end{remark}

\bigskip

\section{Finite-dimensional characterization of the tempered Gamma-grey
Brownian motion}

This section is devoted to the finite dimensional characterization of the
generalized process $B^\theta_{\rho,\alpha}$. We recall that, in order to
overcome the lack of moments, we introduced the tempering factor $\theta$;
we give the following definition of the process in the Euclidean space, in
terms of its $n$-times characteristic function thanks to the $\sigma^*$
algebra.

\begin{definition}
\label{d1} Let $\alpha \in (0,2)$, $\rho \in (0,1]$ and $\theta \geq 0$.
Let, for any $\xi _{k}\in \mathbb{R}$, $k=1,...,n$ and $n\in \mathbb{N},$
\begin{equation}
\Phi _{\alpha ,\rho }^{\theta}(\xi _{1},...,\xi _{n};t_{1},...t_{n})=\frac{%
\Gamma \left( \rho , \theta + \frac{1}{2}\sum_{j,k=1}^{n}\xi _{j}\xi
_{k}\gamma _{\alpha }(t_{j},t_{k}) \right) }{\Gamma (\rho , \theta )},
\label{ll3}
\end{equation}
where $\gamma _{\alpha }(t_{j},t_{k})=t_{k}^{\alpha }+t_{j}^{\alpha
}-|t_{k}-t_{j}|^{\alpha }$ and $0\leq t_{1}\leq ...\leq t_{n}<\infty .$
Then, the process with characteristic function (\ref{ll3}), will be denoted
(as its infinite-dimensional counterpart) as $B_{\alpha ,\rho }^{\theta
}:=\left\{ B_{\alpha ,\rho }^{\theta }(t),t\geq 0\right\} .$
\end{definition}

Thanks to the next result, we can express the tempered $\Gamma$-GBM as a
product of a random variable and a fractional Brownian motion, under the
assumption that they are mutually independent.

\begin{theorem}
For $\alpha \in (0,2)$, $\rho \in (0,1]$ and $\theta>0$, the following
equality of all the finite-dimensional distribution (denoted by $\overset{%
f.d.d.}{=}$) holds
\begin{equation}
B_{\alpha ,\rho }^{\theta }(t)\overset{f.d.d.}{=}\sqrt{Y_{\theta }^{\rho }}%
B^{\alpha /2}(t),\qquad t\geq 0,  \label{ggbm}
\end{equation}
where $B^{\alpha /2}:=\{B^{\alpha /2}(t),t\geq 0\}$ is the fractional
Brownian motion with Hurst-parameter $H=\alpha /2$, for $\alpha \in (0,2)$
and $Y_{\theta}^{\rho }$ is the r.v. with density
\begin{equation}
l^{\theta}_{\rho}(y)=\frac{1}{\Gamma (\rho ,\theta )\Gamma (1-\rho )}\frac{%
e^{-\theta y}}{y(y-1 )^{\rho }}1_{y>1 },\qquad \rho \in (0,1),\theta \geq 0.
\label{ll}
\end{equation}
independent from $B^{\alpha /2}.$
\end{theorem}

\begin{proof}
We have that
\begin{eqnarray*}
&&\mathbb{E}\exp \left\{ i\sqrt{Y_{\theta}^{\rho }} \sum_{k=1}^{n}\xi
_{k}B^{\alpha /2}(t_{k})\right\} \\
&=&\mathbb{E}\left[ \mathbb{E}\left( \left. \exp \left\{ i\sqrt{Y_{\theta
}^{\rho }}\sum_{k=1}^{n}\xi _{k}B^{\alpha /2}(t_{k})\right\} \right\vert
Y_{\theta}^{\rho }\right) \right] \\
&=&\mathbb{E}\left[ \exp \left\{ -\frac{1}{2}Y_{\theta}^{\rho
}\sum_{j,k=1}^{n}\xi _{j}\xi _{k}\gamma _{\alpha }(t_{j},t_{k})\right\} %
\right]
\end{eqnarray*}
which coincides with (\ref{ll3}). This can be proved by taking into account
that
\begin{eqnarray*}
\widetilde{l}_{\theta }^{\rho }(\eta ) &=&\frac{1}{\Gamma (\rho ,\theta
)\Gamma (1-\rho )}\int_{1 }^{+\infty }e^{-(\theta +\eta )y}(y-1)^{-\rho
}y^{-1}dy \\
&=&\frac{e^{-(\theta +\eta )}}{\Gamma (\rho ,\theta )\Gamma (1-\rho )}%
\int_{0}^{+\infty }e^{-(\theta +\eta ) \omega }\omega ^{-\rho }(1+\omega
)^{-1}d\omega \\
&=&[\text{by (1.6.25) in \cite{KIL}}] \\
&=&\frac{e^{-(\theta +\eta )}}{\Gamma (\rho ,\theta )}\Psi (1-\rho ,1-\rho
;\theta +\eta),
\end{eqnarray*}
where $\Psi (a,b;\cdot )$ is the confluent Tricomi hypergeometric function,
together with the well-known relationship $\Psi (a,a;x)=e^{x}\Gamma (1-a;x)$%
, for $a>0$.
\end{proof}

For any $n \in \mathbb{N}$, the joint probability density function of $%
B_{\alpha ,\rho }^{\theta }$ is therefore given by
\begin{equation}
f_{B_{\alpha ,\rho }^{\theta }}(\mathbf{x},\mathbf{\Sigma }_{\alpha })=\frac{
(2\pi )^{-n/2}}{\sqrt{\det \mathbf{\Sigma }_{\alpha }}}\int_{0}^{+\infty
}\tau ^{-n/2}\exp \left\{ -\frac{\mathbf{x}^{T}\mathbf{\Sigma }_{\alpha
}^{-1}\mathbf{x}}{2\tau }\right\} l^{\theta}_{\rho }(\tau )d\tau ,
\label{bb}
\end{equation}
where $\mathbf{\Sigma }_{\alpha }:=(\gamma _{\alpha
}(t_{j},t_{k}))_{j,k=1}^{n}$ and $\mathbf{x} \in \mathbb{R}^n$.

\begin{remark}
It is easy to check that, in the special case where $\rho =1$ and for any $%
\theta$, formula (\ref{ll3}) reduces to the characteristic function of the
fractional Brownian motion with $H=\alpha /2,$ and thus, by adding the
condition $\alpha =1$, we obtain the Brownian motion.
\end{remark}

\begin{remark}
We note that the density (\ref{ll}) coincides with $l^{\rho}_{\theta}(y)=%
\frac{\exp(-\theta y)}{f_{\rho}(y)}$, where $f_{\rho}(y)$ is given in (\ref%
{kl}), as can be easily checked by considering property P2 in Appendix A and (2.9.6)
in \cite{KIL2}.
\end{remark}

\begin{theorem}
Let $\alpha \in (0,2)$, $\rho \in (0,1)$ and $\theta >0$. The $k$-th order
moment of the tempered $\Gamma$-Grey Brownian Motion is given by
\begin{equation}
\mathbb{E}\left[ B_{\alpha ,\rho }^{\theta }(t)^{k}\right] =\left\{
\begin{array}{l}
0,\qquad k=2n+1 \\
\frac{2 t^{\alpha n}}{\Gamma (\rho , \theta )}{\LARGE G} _{1,2}^{2,0}\left[
\left. \theta \right\vert
\begin{array}{cc}
1-n &  \\
0, & \rho -n%
\end{array}
\right] ,\qquad k=2n%
\end{array}
\right.  \label{ms}
\end{equation}
for $k,n\in \mathbb{N}$, while its autocovariance reads
\begin{equation}
cov(B_{\alpha ,\rho }^{\theta }(t),B_{\alpha ,\rho }^{\theta }(s)) =\frac{
e^{-\theta}\theta ^{\rho -1}}{\Gamma (\rho , \theta )}\left[ t^{\alpha
}+s^{\alpha }-|t-s|^{\alpha }\right] .  \label{cov}
\end{equation}
\end{theorem}

\begin{proof}
We first evaluate the $k$-th order moment of the r.v. $Y_{\theta }^{\rho }$,
for $k\in \mathbb{N}$, as follows
\begin{eqnarray}
\mathbb{E}\left[ \left( Y_{\theta }^{\rho }\right) ^{k}\right] &=& \frac{1}{%
\Gamma (\rho , \theta ) \Gamma(1-\rho)}\int_{1 }^{+\infty }
y^{k-1}(y-1)^{-\rho}e^{-\theta y} dy  \label{yy} \\
&=&[\text{by (2.9.36) in \cite{KIL2}}]  \notag \\
&=&\frac{1}{\Gamma (\rho , \theta )}{\LARGE G}_{1,2}^{2,0} \left[ \left.
\theta \right\vert
\begin{array}{cc}
1-k &  \\
0, & \rho -k%
\end{array}
\right] .  \notag
\end{eqnarray}
By considering (\ref{ggbm}), together with the expression of the $k$-moment
of the fractional Brownian motion, formula (\ref{ms}) easily follows from
the independence between $B^{\alpha /2}$ and $Y_{\theta }^{\rho }.$ The
autocovariance can be obtained as follows
\begin{eqnarray*}
cov(B_{\alpha ,\rho }^{\theta }(t),B_{\alpha ,\rho }^{\theta }(s)) &=&
\mathbb{E}\left( Y_{\theta }^{\rho }\right) \mathbb{E}\left( B^{\alpha
/2}(t)\cdot B^{\alpha /2}(s)\right) \\
&=&\frac{1 }{\Gamma (\rho , \theta )}{\LARGE G}_{1,2}^{2,0} \left[ \left.
\theta \right\vert
\begin{array}{cc}
0 &  \\
0, & \rho -1%
\end{array}
\right] \left[ t^{\alpha }+s^{\alpha }-|t-s|^{\alpha }\right] \\
&=&\frac{1 }{\Gamma (\rho , \theta )}\frac{1}{2\pi i}\int_{ \mathcal{L}}%
\frac{\Gamma(h)\Gamma (\rho -1+h)\theta^{-h}}{ \Gamma (h)}dh \left[
t^{\alpha }+s^{\alpha }-|t-s|^{\alpha }\right] \\
&=&\frac{1 }{\Gamma (\rho , \theta )}{\LARGE H}_{0,1}^{1,0} \left[ \left.
\theta \right\vert
\begin{array}{c}
- \\
(\rho -1,1)%
\end{array}
\right] \left[ t^{\alpha }+s^{\alpha }-|t-s|^{\alpha }\right] ,
\end{eqnarray*}
which coincides with (\ref{cov}), by taking into account (1.125) in \cite%
{MAT}.
\end{proof}

\begin{remark}
For $\rho =1$ and any $\theta$, formula (\ref{cov}) reduces to the
covariance of the fractional Brownian motion.
\end{remark}

Finally, from (\ref{ll3}) it is clear that the process $B_{\alpha ,\rho
}^{\theta }$ has stationary increments with characteristic function
\begin{equation}
\mathbb{E}\exp \left\{ i\xi \lbrack B_{\alpha ,\rho }^{\theta }(t)-B_{\alpha
,\rho }^{\theta }(s)]\right\} =\frac{\Gamma \left( \rho , \theta + \frac{\xi
^{2}}{2}|t-s|^{\alpha } \right) }{\Gamma (\rho , \theta )},\quad \xi \in
\mathbb{R},\text{ }t,s\geq 0.  \label{ggb2}
\end{equation}

\section{Time-change representation of the Gamma-grey Brownian motion}

In this section we present a characterization of $B^{\theta}_{\alpha,\rho}$
as a time-changed Brownian motion, that holds in the sense of the
one-dimensional distribution and in the special case where $\theta=0$.

\subsection{The random-time process}

We start by introducing the following process that will represent the
random-time argument.

\begin{definition}
Let $Y_{\rho }(t),$ $t\geq 0$, be the stochastic process defined by means of
the following Laplace transform of its $n$-times density
\begin{equation}
\mathbb{E}e^{-\sum_{k=1}^{n}\eta _{k}Y_{\rho }(t_{k})}=\frac{ \Gamma \left(
\rho , \sum_{k=1}^{n}\eta _{k}t_{k}\right) }{\Gamma \left( \rho \right) }%
,\qquad \eta _{1},...\eta _{n}>0,\text{ }\rho \in (0,1).  \label{chi}
\end{equation}
\end{definition}

The previous definition is well-posed, since the function (\ref{chi}) can be
checked to be completely monotone (w.r.t. $\eta _{1},...\eta _{n}$ and for
any choice of $t_{1},...t_{n}\geq 0$), by adapting the result of Lemma 3.1
in \cite{BEG} \ to the case $\alpha =1.$ The process is, by definition,
self-similar with scaling parameter equal to one, since, by (\ref{chi}), we
get that $\left\{ aY^{\rho}(t),t\geq 0\right\} \overset{f.d.d.}{=} \left\{
Y_{\rho }(at),t\geq 0\right\} ,$ for any $a>0.$ Moreover, it has stationary
increments, as can be seen by taking into account that (\ref{chi}) is
well-defined even for $\eta _{j}<0,$ for any $j$ (by analytic continuation),
so that we have that
\begin{equation}
\mathbb{E}e^{-\eta \lbrack Y_{\rho }(t_{2})-Y_{\rho }(t_{1})]}=\frac{\Gamma
\left( \rho , \eta (t_{2}-t_{1})\right) }{ \Gamma \left( \rho \right) }%
,\qquad t_{2}>t_{1}\geq 0.  \label{chic}
\end{equation}
We denote by $l_{\rho }(y,t)$
the transition density of $Y_{\rho }(t),$ (for $y,t\geq 0,$ $\rho \in (0,1)$); therefore, as a
consequence of the self-similarity, we have that $l_{\rho
}(y,t)=t^{-1}l_{\rho }(yt^{-1})$ (where $l_{\rho }(\cdot )$ is given in (\ref%
{ll}), with $\theta =0$) and
\begin{equation}
l_{\rho }(y,t)=\frac{1}{\Gamma (\rho)\Gamma (1-\rho )}\frac{1_{y>t }}{y(y-1
)^{\rho }},\qquad \rho \in (0,1).  \label{cc}
\end{equation}
Its space-Laplace transform coincides with (\ref{chi}), for $n=1,$ i.e. $
\widetilde{l}_{\rho }(\eta ,t)=\Gamma (\rho ,\eta t)/\Gamma (\rho ).$

Formula (\ref{chic}) proves also that $Y_{\rho }$ has increasing
trajectories, since, by (\ref{cc}), we have that $Y_{\rho }(t_{2})-Y_{\rho
}(t_{1})\geq t_{2}-t_{1}$ almost surely. 

We recall that, in the ggBm case, the random-time argument is represented by
the inverse of the $\beta$-stable subordinator. Therefore, e are interested
in checking if, also in the $\Gamma$-GBM case, it is possible to define the
random-time argument as the inverse of another stochastic process and to
characterize the latter.

By resorting to the Doob's theorem, we can refer to the separable version of
$Y_{\rho }$ so that its hitting time is well-defined as follows
\begin{equation}
T_{\rho }(x):=\inf \{t\geq 0:Y_{\rho }(t)>x\},\qquad x\geq 0.  \label{ggb3}
\end{equation}
We now derive its transition density.

\begin{theorem}
The space-Laplace transform of the density of the process $T_{\rho }$
defined in (\ref{ggb3}) is given by
\begin{equation}
\mathbb{E}e^{-\xi T_{\rho }(x)}=E_{1,1}^{\rho }\left( - \xi x \right)
,\qquad \xi >0,\text{ }\rho \in (0,1),\text{ }x\geq 0, \text{ }  \label{ecc}
\end{equation}
and its transition density $h_{\rho }(t,x):=P\{T_{\rho }(x)\in dt\}/dt$, $%
t,x\geq 0,$ reads
\begin{equation}
h_{\rho }(t,x)=\frac{t^{\rho -1}(x- t)^{-\rho }1_{t<x }}{\Gamma (\rho
)\Gamma (1-\rho )}.  \label{ecc2}
\end{equation}
\end{theorem}

\begin{proof}
By considering (\ref{ggb3}) we can write that $P\{Y_{\rho
}(t)>x\}=P\{T_{\rho }(x)<t\},$ so that, taking the Laplace transform w.r.t. $%
x$ and denoting by $\gamma \left( \rho ,x\right) =\int_{0}^{x}e^{-w}w^{\rho
-1}dw$ the lower incomplete gamma function, we have that
\begin{eqnarray*}
\int_{0}^{t}\widetilde{h}_{\rho }(z,\eta )dz &=&\int_{0}^{+\infty }e^{-\eta
x}\int_{x}^{+\infty }l_{\rho }(z,t)dzdx=\frac{1}{\eta }\int_{0}^{+\infty
}(1-e^{-\eta z})l_{\rho }(z,t)dz \\
&=&\frac{\Gamma \left( \rho \right) -\Gamma \left( \rho , \eta t\right) }{%
\eta \Gamma \left( \rho \right) }=\frac{\gamma \left( \rho , \eta t\right) }{%
\eta \Gamma \left( \rho \right) } \\
&=&\frac{1}{\eta \Gamma \left( \rho \right) }\int_{0}^{ \eta t}e^{-w}w^{\rho
-1}dw.
\end{eqnarray*}
By taking also the Laplace transform w.r.t. $t$, we get
\begin{equation}
\widetilde{\widetilde{h}}_{\rho }(\xi ,\eta )=\frac{\eta ^{\rho -1}}{\left(
\xi+\eta \right) ^{\rho }},  \label{ps}
\end{equation}
whose inverse transform (w.r.t. $\eta $) coincides with (\ref{ecc}). It is
easy to check that the the inverse Laplace transform (w.r.t. $\xi )$ of the
latter reads
\begin{equation*}
h_{\rho }(t,x)=\frac{1_{t<x }}{t\Gamma (\rho )}{\LARGE G}_{1,1}^{1,0} \left[
\left. \frac{ t}{x}\right\vert
\begin{array}{c}
1 \\
\rho%
\end{array}
\right] .
\end{equation*}
Indeed, by taking into account (\ref{mei}) together with formula (2.19) in
\cite{MAT} (since $\rho >0)$ we get
\begin{eqnarray*}
\int_{0}^{+\infty }e^{-\xi t}h_{\rho }(t,x)dt &=&\frac{1}{\Gamma (\rho )}
\int_{0}^{+\infty }\frac{e^{-\xi t}}{t}{\LARGE H}_{1,1}^{1,0}\left[ \left.
\frac{ t}{x}\right\vert
\begin{array}{c}
(1,1) \\
(\rho ,1)%
\end{array}
\right] dt \\
&=&\frac{1}{\Gamma (\rho )}{\LARGE H}_{2,1}^{1,1}\left[ \left. \frac{1 }{%
x\xi }\right\vert
\begin{array}{c}
(1,1)(1,1) \\
(\rho ,1)%
\end{array}
\right] \\
&=&[\text{by property P1 in Appendix A}] \\
&=&\frac{1}{\Gamma (\rho )}{\LARGE H}_{1,2}^{1,1}\left[ \left. x\xi
\right\vert
\begin{array}{c}
(1-\rho ,1) \\
(0,1)(0,1)%
\end{array}
\right] ,
\end{eqnarray*}
which coincides with (\ref{ecc}) by (1.137) in \cite{MAT}. Moreover, by
resorting to formulae (2.4)-(2.5) in \cite{MAI} and by property P2 in Appendix A, we
can simplify the previous expression into (\ref{ecc2}).
\end{proof}

\

It is immediate to see from (\ref{ecc}), that, as happens for $Y_{\rho },$ also $%
T_{\rho }$ is self-similar, with scaling parameter equal to one, since $%
\left\{ aT_{\rho }(t),t\geq 0\right\} \overset{f.d.d.}{=}\left\{ T_{\rho
}(at),t\geq 0\right\} ,$ for any $a>0.$ Moreover, the following relationship
holds between the densities $h_{\rho }(t,x)$ and $l_{\rho }(x,t),$ of $%
T_{\rho }$ and $Y_{\rho }$, respectively:
\begin{equation}
h_{\rho }(t,x)=\frac{x}{t}l_{\rho }(x,t),\qquad x,t\geq 0.  \label{rel}
\end{equation}

It is easy to derive the partial differential equations (p.d.e.'s) satisfied
by the densities of $T_{\rho },$ given in (\ref{ecc2}), and of its inverse $%
Y_{\rho },$ given in (\ref{cc}); for this reason, we omit the proof of the
following result.

\begin{coro}
The density of the process $T_{\rho }$ satisfies the following p.d.e.
\begin{equation}
\frac{\partial }{\partial t}h_{\rho }(t,y)=-\frac{\partial }{ \partial y}%
h_{\rho }(t,y)+\frac{\rho -1}{t}h_{\rho }(t,y),\quad t,y\geq 0,  \label{hh}
\end{equation}
with initial condition $h_{\rho }(t,0)=\delta (t),$ while the density of $
Y_{\rho }$ satisfies the following p.d.e.
\begin{equation}
\frac{\partial }{\partial t}l_{\rho }(y,t)=-\frac{\partial }{ \partial y}%
l_{\rho }(y,t)+\left[ \frac{\rho }{t}-\frac{1 }{y}\right] l_{\rho
}(y,t),\quad t,y\geq 0,  \label{hh2}
\end{equation}
with initial condition $l_{\rho }(y,0)=0.$
\end{coro}

\begin{remark}
For $\rho =1$, equation (\ref{hh}) reduces to the partial differential
equation satisfied by the density of the elementary subordinator $
T_{1}(y)=y ,$ which is equal to $h_{1}(t,y)=\delta (t-y),$ as can be easily
checked by taking the Laplace transform w.r.t. $y.$ Analogously, we have
that $l_{1}(y,t)=t\delta (t-y)/y$, which satisfies equation (\ref{hh2}),
with $\rho =1$. Another interesting special case is for $\rho =1/2$. In this
case the densities of the processes $T_{1/2} $ and $Y_{1/2}$ are
respectively equal to
\begin{equation}
h_{1/2}(t,y)=\frac{1}{\pi \sqrt{t(y- t)}}1_{t<y },  \label{ar}
\end{equation}
which coincides with the arcsine law, and
\begin{equation}
l_{1/2}(y,t)=\frac{\sqrt{ t}}{\pi y\sqrt{y-t}}1_{y>t}.  \label{ar2}
\end{equation}
\end{remark}

\begin{remark}
We report in the following table, for the reader's convenience, the Laplace
pairs (w.r.t. time and space) of the densities $h_{\rho }(t,y)$ and $l_{\rho
}(y,t).$

\
\begin{equation*}
\begin{tabular}{|l|l|l|}
\hline
Process $T_{\rho }$ with density $h_{\rho }(t,y)$ & $\widetilde{h} _{\rho
}(\xi ,y)=E_{1,1}^{\rho }(-\xi y)$ & $\widetilde{h}_{\rho }(t,\eta )=\frac{1%
}{\Gamma (\rho )}(\eta t)^{\rho -1}e^{-\eta t}$ \\ \hline
Process $Y_{\rho }$ with density $l_{\rho }(y,t)$ & $\widetilde{l} _{\rho
}(\eta ,t)=\frac{\Gamma (\rho ,\eta t)}{\Gamma (\rho )}$ & $\widetilde{l}%
_{\rho }(y,\xi )=\rho E_{1,2}^{1+\rho }(-y\xi )$ \\ \hline
\end{tabular}%
\end{equation*}

\

The time-Laplace transform of $l_{\rho }(y,t)$\ can be obtained taking into
account (1.6.15) and (1.9.3) in \cite{KIL}. It is evident from the previous
corollary that, despite the expression of the Laplace transform of $h_{\rho
}(t,y)$ and $l_{\rho }(y,t)$ (w.r.t. space and time, respectively) is given
in terms of Mittag-Leffler functions, the p.d.e. governing the densities of
both $Y_{\rho }$ and $T_{\rho }$ do not involve fractional operators.
\end{remark}

\subsection{Time-changed representation and governing equation}

\label{OneDimRep}

We start by considering the time-change of a standard Brownian motion $%
B:=\{B(t),t>0\}$ by the time-stretched process $Y_{\rho }(t^{\alpha }),$
under the assumption that the latter is independent of $B$, i.e.
\begin{equation}  \label{gamma}
B_{\alpha,\rho}(t):=B(Y_{\rho }(t^{\alpha })),
\end{equation}
for $\rho \in (0,1]$ and $\alpha \in (0,1]$.\newline
As a consequence of (\ref{chi}), we can write its (one-dimensional)
characteristic function as
\begin{equation}
\Phi _{\alpha ,\rho }(\xi ,t):=\mathbb{E}e^{i\xi B_{\alpha, \rho}(t)}=\frac{
\Gamma \left( \rho , \xi ^{2}t^{\alpha }/2\right) }{\Gamma (\rho )},
\label{ft}
\end{equation}
from which it is immediate to check that the following equality of the
one-dimensional distribution (hereafter denoted by $\overset{d}{=}$) holds
\begin{equation}
B_{\alpha ,\rho }(t)\overset{d}{=}\sqrt{Y_{\rho }}B^{\alpha /2}(t),\qquad
t\geq 0,  \label{mol}
\end{equation}
where $B^{\alpha /2}:=\{B^{\alpha /2}(t),t\geq 0\}$ is a fractional Brownian
motion and $Y_{\rho }$ is a r.v., independent of $B^{\alpha /2}$, with
density $l_{\rho }(y)$ given in (\ref{ll}), with $\theta =0$. We note that
the moments of any order of $B_{\alpha ,\rho }$ are infinite, as can be
easily checked by considering (\ref{yy}). \newline

\begin{remark}
It is well-known that, in the case of the ggBm $B^{\beta ,\alpha
}:=\{B^{\beta ,\alpha }(t),t\geq 0\}$, the following equality of the
one-dimensional distribution holds $B^{\beta ,\alpha }(t)\overset{d}{=}
B(X^{\beta }(t^{\alpha /\beta })),$ $t\geq 0,$ where the random time
argument $X^{\beta }:=\left\{ X^{\beta }(t),t\geq 0\right\} ,$ is the
inverse of a stable subordinator of index $\beta \in (0,1)$ (see \cite{DAS}
and \cite{MUR2}). As remarked in \cite{DAS} for the ggBm, also in this case
the representation (\ref{mol}) holds only for the one-dimensional
distribution. Indeed, for example, the two-times characteristic function of $%
B(Y_{\rho }(t^{\alpha }))$ reads:
\begin{equation*}
\mathbb{E}e^{i\xi _{1}B(Y_{\rho }(t_{1}^{\alpha }))+i\xi _{2}B(Y_{\rho
}(t_{2}^{\alpha }))}=\frac{\Gamma \left( \rho , (\xi _{1}^{2}+\xi _{1}\xi
_{2})t_{1}^{\alpha }+ (\xi _{2}^{2}+\xi _{1}\xi _{2})t_{2}^{\alpha }\right)
}{\Gamma \left( \rho \right) },
\end{equation*}%
and therefore it does not depend on $|t_{1}-t_{2}|^{\alpha }$, on the
contrary of what happens for $\mathbb{E}e^{i\xi _{1}B_{\alpha
,\rho}(t_{1})+i\xi _{2}B_{\alpha ,\rho }(t_{2})}$ (as can be easily seen
from formula (\ref{ll3}), for $n=2$ and $\theta =0$).
\end{remark}

Note that, in our case, the stretching effect of time is obtained by the
power of $\alpha$, and does not depend on $\rho$. This affects also the
following governing equation. We prove now that the characteristic function
of $B_{\alpha ,\rho }$ satisfies a time-stretched integral equation, in
analogy with the ggBm.

\begin{theorem}
Let $\rho \in (0,1]$ and $\alpha \in (0,1]$. Let $e_{\rho }^{z}:=z^{\rho
-1}E_{\rho ,\rho }(z^{\rho })$ be the so-called $\rho $-exponential function
(see \cite{KIL}, p.50), then the characteristic function (\ref{ft})
satisfies the following integral equation
\begin{equation}
\Phi _{\alpha ,\rho }(\xi ,t)=1-\frac{\alpha \xi ^{2}}{2} \int_{0}^{t}e^{-%
\frac{ \xi ^{2}}{2}(t^{\alpha }-s^{\alpha })}e_{\rho }^{\frac{ \xi ^{2}}{2}%
(t^{\alpha }-s^{\alpha })}s^{\alpha -1}\Phi _{\alpha ,\rho }(\xi
,s)ds,\qquad t\geq 0,\xi \in \mathbb{R}\text{.}  \label{ll4}
\end{equation}
\end{theorem}

\begin{proof}
Let, for simplicity, $A:=\xi ^{2}/2,$ then we rewrite the integral in the
r.h.s. of (\ref{ll4}) as
\begin{eqnarray*}
&&\frac{1}{\Gamma (\rho )}\int_{0}^{t}e^{-A(t^{\alpha }-s^{\alpha })}\left(
A(t^{\alpha }-s^{\alpha })\right) ^{\rho -1}E_{\rho ,\rho }(A^{\rho
}(t^{\alpha }-s^{\alpha })^{\rho })s^{\alpha -1}\Gamma \left( \rho
,As^{\alpha }\right) ds \\
&=&[s=tw^{1/\alpha }] \\
&=&\frac{A^{\rho -1}t^{\alpha \rho }}{\alpha \Gamma (\rho )}
\int_{0}^{1}e^{-At^{\alpha }(1-w)}(1-w)^{\rho -1}E_{\rho ,\rho }(A^{\rho
}t^{\alpha \rho }(1-w)^{\rho })\Gamma \left( \rho ,At^{\alpha }w\right) dw.
\end{eqnarray*}
Thus the r.h.s. of (\ref{ll4}) reads
\begin{eqnarray*}
&&1-\frac{A^{\rho }t^{\alpha \rho }}{\Gamma (\rho )}\int_{0}^{1}e^{-At^{
\alpha }(1-w)}(1-w)^{\rho -1}E_{\rho ,\rho }(A^{\rho }t^{\alpha \rho
}(1-w)^{\rho })\Gamma \left( \rho ,At^{\alpha }w\right) dw \\
&=&[\text{by (\ref{gg})}] \\
&=&1-A^{\rho }t^{\alpha \rho }\int_{0}^{1}e^{-At^{\alpha }y}y^{\rho
-1}\sum_{j=0}^{\infty }\frac{(A^{\rho }t^{\alpha \rho }y^{\rho })^{j}}{
\Gamma (\rho j+\rho )}dy+ \\
&&+A^{\rho }t^{\alpha \rho }\int_{0}^{1}e^{-At^{\alpha }(1-w)}(1-w)^{\rho
-1}A^{\rho }t^{\alpha \rho }w^{\rho }e^{-At^{\alpha }w}\sum_{j=0}^{\infty }
\frac{(A^{\rho }t^{\alpha \rho }(1-w)^{\rho })^{j}}{\Gamma (\rho j+\rho )}
\sum_{l=0}^{\infty }\frac{(At^{\alpha }w)^{l}}{\Gamma (\rho +l+1)}dy \\
&=&1-A^{\rho }t^{\alpha \rho }\sum_{j=0}^{\infty }\frac{(A^{\rho }t^{\alpha
\rho })^{j}}{\Gamma (\rho j+\rho )}\sum_{l=0}^{\infty }\frac{
(-1)^{l}(At^{\alpha })^{l}}{l!}\int_{0}^{1}y^{l+\rho j+\rho -1}dy+ \\
&&+A^{2\rho }t^{2\alpha \rho }e^{-At^{\alpha }}\sum_{j=0}^{\infty }\frac{
(A^{\rho }t^{\alpha \rho })^{j}}{\Gamma (\rho j+\rho )}\sum_{l=0}^{\infty }
\frac{(At^{\alpha })^{l}}{\Gamma (\rho +l+1)}\frac{\Gamma (\rho +l+1)\Gamma
(\rho j+\rho )}{\Gamma (\rho j+2\rho +l+1)} \\
&=&1-\sum_{j=0}^{\infty }\frac{1}{\Gamma (\rho j+\rho )}\sum_{l=0}^{\infty }
\frac{(-1)^{l}(At^{\alpha })^{l+\rho j+\rho }}{l!(l+\rho j+\rho )}
+\sum_{j=0}^{\infty }e^{-At^{\alpha }}\sum_{l=0}^{\infty }\frac{(At^{\alpha
})^{\rho j+2\rho +l}}{\Gamma (\rho j+2\rho +l+1)} \\
&=&1-\sum_{j=0}^{\infty }\frac{\Gamma (\rho j+\rho )-\Gamma (\rho j+\rho
,At^{\alpha })}{\Gamma (\rho j+\rho )}+\sum_{j=0}^{\infty }\frac{\Gamma
(\rho j+2\rho )-\Gamma (\rho j+2\rho ,At^{\alpha })}{\Gamma (\rho j+2\rho )}
\end{eqnarray*}
where in the last step, we have applied (\ref{gg2}) for the second term and (%
\ref{gg}) for the last one. After a change of index in the second sum, we
easily obtain (\ref{ft}).
\end{proof}

\

Equation (\ref{ll4}) reduces, for $\rho =1$, to the equation satisfied by
the characteristic function of the fractional Brownian Motion, i.e.
\begin{equation}
\frac{\partial }{\partial t}u(\xi ,t)=-\frac{\alpha }{2}t^{\alpha -1}\xi
^{2}u(\xi ,t).  \label{fbm}
\end{equation}
On the other hand, for $\rho <1,$ it can be compared with that satisfied by
the characteristic function of the ggBm (see Proposition 4.1 in \cite{MUR});
in this case the presence of the variable $\xi $ also in the integral's
kernel does not allow to obtain, by the Fourier inversion, a master equation
for the density of the process, as happens for the ggBm.

We then provide an alternative result, which leads to the governing equation
of the marginal density of $B_{\alpha ,\rho }$. In this case, we will resort
to the equality in distribution (\ref{mol}).

\begin{theorem}
Let $\rho \in (0,1]$ and $\alpha \in (0,1]$. The density
\begin{equation}
f_{B_{\alpha ,\rho }}(x,t)=\frac{1}{\sqrt{4\pi t^{\alpha }}}
\int_{0}^{+\infty }\tau ^{-1/2}\exp \left\{ -\frac{x^{2}}{4\tau t^{\alpha }}
\right\} l_{\rho }(\tau )d\tau ,\qquad x\in \mathbb{R},\;t\geq 0,
\label{bb2}
\end{equation}
where $l_{\rho }(\cdot )$ is given in (\ref{ll})$,$ satisfies the following
integro-differential equation
\begin{equation}
\frac{\partial }{\partial t}f(x,t)=\frac{\alpha \rho }{t}\left[
f(x,t)-f(x,0) \right] +\frac{ \alpha }{2t}\frac{\partial ^{2}}{\partial x^{2}%
} \int_{0}^{t}z^{\alpha }\frac{\partial }{\partial z}f(x,z)dz,  \label{bb3}
\end{equation}
with initial conditions $f(x,0)=\delta (x)$ and $f(0,t)=0.$
\end{theorem}

\begin{proof}
Let $\widehat{f}(\xi ):=\int_{-\infty }^{+\infty }e^{i\xi x}f(x)dx$ denote
the Fourier transform, then by transforming (\ref{bb3}), w.r.t. $x,$ and
considering (\ref{ft}), we can check that $\Phi _{\alpha ,\rho }(\xi ,t)$
satisfies the following equation
\begin{equation}
\frac{\partial }{\partial t}\widehat{f}(\xi ,t)=\frac{\alpha \rho }{t}\left[
\widehat{f}(\xi ,t)-1\right] -\frac{\alpha }{2}t^{\alpha -1}\xi ^{2}
\widehat{f}(\xi ,t)+\frac{ \alpha ^{2}\xi ^{2}}{2t} \int_{0}^{t}z^{\alpha -1}%
\widehat{f}(\xi ,z)dz,  \label{bb4}
\end{equation}
where we have taken into account the initial condition, $\widehat{f}(\xi
,0)=1$. We then rewrite the r.h.s. of (\ref{bb4}) as follows:
\begin{eqnarray}
&&\frac{\alpha \rho }{t}\left[ \frac{\Gamma \left( \rho , \xi ^{2}t^{\alpha
}/2\right) }{\Gamma (\rho )}-1\right] -\frac{1 }{2} \alpha t^{\alpha -1}\xi
^{2}\frac{\Gamma \left( \rho , \xi ^{2}t^{\alpha }/2\right) }{\Gamma (\rho )}%
+\frac{ \alpha ^{2}\xi ^{2} }{2t}\int_{0}^{t}z^{\alpha -1}\frac{\Gamma
\left( \rho , \xi ^{2}z^{\alpha }/2\right) }{\Gamma (\rho )}dz  \label{bb5}
\\
&=&-\frac{\alpha \rho }{t\Gamma (\rho )}\gamma \left( \rho , \xi
^{2}t^{\alpha }/2\right) -\frac{1 }{2}\alpha t^{\alpha -1}\xi ^{2} \frac{%
\Gamma \left( \rho , \xi ^{2}t^{\alpha }/2\right) }{\Gamma (\rho )}+\frac{
\alpha \xi ^{2}}{2t\Gamma (\rho )}\int_{0}^{t^{\alpha }}\int_{ \xi
^{2}w/2}^{+\infty }e^{-y}y^{\rho -1}dydw  \notag \\
&=&-\frac{\alpha \rho }{t\Gamma (\rho )}\gamma \left( \rho , \xi
^{2}t^{\alpha }/2\right) -\frac{1}{2}\alpha t^{\alpha -1}\xi ^{2} \frac{%
\Gamma \left( \rho ,\xi ^{2}t^{\alpha }/2\right) }{\Gamma (\rho )}+\frac{%
\alpha }{t\Gamma (\rho )}\int_{0}^{ \xi ^{2}t^{\alpha }/2}e^{-y}y^{\rho }dy+
\notag \\
&&+\frac{ \alpha \xi ^{2}t^{\alpha -1}}{2\Gamma (\rho )}\int_{ \xi
^{2}t^{\alpha }/2}^{+\infty }e^{-y}y^{\rho -1}dy  \notag \\
&=&-\frac{\alpha }{t\Gamma (\rho )}\left[ \gamma \left( \rho +1, \xi
^{2}t^{\alpha }/2\right) \right] -\frac{\alpha }{t\Gamma (\rho )}\left(
\frac{ \xi ^{2}t^{\alpha }}{2}\right) ^{\rho }e^{-\xi ^{2}t^{\alpha }/2}+%
\frac{\alpha }{t\Gamma (\rho )}\gamma \left( \rho +1,\xi ^{2}t^{\alpha
}/2\right) ,  \notag
\end{eqnarray}
where we have applied the following well-known relationship between upper
incomplete and lower incomplete gamma functions $\Gamma \left( \rho \right)
=\Gamma \left( \rho ,x\right) +\gamma \left( \rho ,x\right) $ and the
recurrence formula
\begin{equation*}
\gamma \left( \rho +1,x\right) =\rho \gamma \left( \rho ,x\right) -x^{\rho
}e^{-x}
\end{equation*}
(see \cite{GRA}, p.951). It is now easy to check that (\ref{bb5}) coincides
with the first derivative of $\Phi _{\alpha ,\rho }(\xi ,t)$ (w.r.t. $t$)
and then equation (\ref{bb4}) holds.
\end{proof}

\

The knowledge of the p.d.e. governing the density (\ref{bb2}) can be then
used in order to simulate the trajectories of the tempered $\Gamma$-GBM.
Note that, for $\rho =1,$ equation (\ref{bb4}) reduces to (\ref{fbm}),
since, in this case $\widehat{f}(\xi ,t)=e^{-\xi ^{2}t^{\alpha }/2}$ and
thus $\int_{0}^{t}z^{\alpha -1}\widehat{f}(\xi ,z)dz=t^{\alpha }\widehat{f}%
(\xi ,t)/\alpha .$

\appendix

\section{Special functions}

We present here some definitions and results concerning special functions
that are needed in our analysis.

Let us recall the definition of the $H$\emph{-function} (see, for example,
\cite{MAT}, p.21):%
\begin{equation}
{\LARGE H}_{p,q}^{m,n}\left[ \left. z\right\vert
\begin{array}{c}
(a_{p},A_{p}) \\
(b_{q},B_{q})%
\end{array}%
\right] :=\frac{1}{2\pi i}\int_{\mathcal{L}}\frac{\left\{
\prod\limits_{j=1}^{m}\Gamma (b_{j}+B_{j}s)\right\} \left\{
\prod\limits_{j=1}^{n}\Gamma (1-a_{j}-A_{j}s)\right\} z^{-s}ds}{\left\{
\prod\limits_{j=m+1}^{q}\Gamma (1-b_{j}-B_{j}s)\right\} \left\{
\prod\limits_{j=n+1}^{p}\Gamma (a_{j}+A_{j}s)\right\} },
\end{equation}%
with $z\neq 0,$ $m,n,p,q\in \mathbb{N}_{0},$ for $0\leq m\leq q$, $0\leq
n\leq p$, $a_{j},b_{j}\in \mathbb{R},$ for $i=1,...,p,$ $j=1,...,q$ and $%
\mathcal{L}$ is a contour such that the following condition is satisfied%
\begin{equation}
(b_{j}+\alpha )\neq (a_{l}-k-1),\qquad j=1,...,m,\text{ }l=1,...,n,\text{ }%
\alpha ,k=0,1,...  \label{1.6}
\end{equation}%
We need the following well-known properties of the H-function.

\begin{enumerate}[label=\textbf{P\arabic*}]
\item For any $z\neq 0,$ we have that%
\begin{equation}
{\LARGE H}_{p,q}^{m,n}\left[ \left. z\right\vert
\begin{array}{c}
(a_{p},A_{p}) \\
(b_{q},B_{q})%
\end{array}%
\right] ={\LARGE H}_{q,p}^{n,m}\left[ \left. \frac{1}{z}\right\vert
\begin{array}{c}
(1-b_{q},B_{q}) \\
(1-a_{p},A_{q})%
\end{array}%
\right]  \notag
\end{equation}%
(see equation (1.58) in \cite{MAT}).

\item For any $\sigma \in \mathbb{C},$ we have that%
\begin{equation}
z^{\sigma }{\LARGE H}_{p,q}^{m,n}\left[ \left. z\right\vert
\begin{array}{c}
(a_{p},A_{p}) \\
(b_{q},B_{q})%
\end{array}%
\right] ={\LARGE H}_{p,q}^{m,n}\left[ \left. \frac{1}{z}\right\vert
\begin{array}{c}
(a_{p}+\sigma A_{p},A_{p}) \\
(b_{q}+\sigma B_{q},B_{q})%
\end{array}%
\right]  \notag
\end{equation}%
(see equation (1.60) in \cite{MAT}).
\end{enumerate}

We recall that the \emph{Meijer G-function} is a special case of the
H-function (see \cite{KAR}), i.e.
\begin{equation}
{\LARGE G}_{p,q}^{m,n}\left[ \left. z\right\vert
\begin{array}{c}
(a_{1},...a_{p}) \\
(b_{1},...,b_{q})%
\end{array}%
\right] ={\LARGE H}_{p,q}^{m,n}\left[ \left. z\right\vert
\begin{array}{ccc}
(a_{1},1) & ... & (a_{p},1) \\
(b_{1},1) & ... & (b_{q},1)%
\end{array}%
\right] ,  \label{mei}
\end{equation}%
and that the function ${\LARGE G}_{p,p}^{p,0}\left[ \left. x\right\vert \;%
\right] $ vanishes, for any $|x|>1$, $p\in \mathbb{N}$ (see \cite{KAR},
Property 3).

\

Let us consider the \emph{upper-incomplete gamma function}, defined as $%
\Gamma (\rho ,x):=\int_{x}^{+\infty }e^{-t}t^{\rho -1}dt.$ We recall its
following series representations%
\begin{equation}
\Gamma (\rho ,x)=\Gamma (\rho )\left( 1-x^{\rho }e^{-x}\sum_{j=0}^{\infty }%
\frac{x^{j}}{\Gamma (\rho +j+1)}\right)  \label{gg}
\end{equation}%
and
\begin{equation}
\Gamma (\rho ,x)=\Gamma (\rho )-\sum_{j=0}^{\infty }\frac{(-1)^{j}x^{\rho +j}%
}{j!(\rho +j)},  \label{gg2}
\end{equation}%
for $x>0$ and $\rho \neq 0,-1,-2,...$ (see \cite{ABR}).

\

Finally, we recall the definition of the \emph{Mittag-Leffler function with
three parameters} (also called Prabhakar function), for any $x\in \mathbb{C}$%
,%
\begin{equation*}
E_{\alpha ,\beta }^{\gamma }\left( x\right) :=\sum_{j=0}^{\infty }\frac{%
(\gamma )_{j}x^{j}}{j!\Gamma (\alpha j+\beta )},\qquad \alpha ,\beta ,\gamma
\in \mathbb{C}\text{, }\func{Re}(\alpha )>0,
\end{equation*}%
where $(\gamma )_{j}:=\Gamma (\gamma +j)/\Gamma (\gamma ),$ together with
the $n$-order differentiation formula (see \cite{GOR} and \cite{GAR}, for
details), for any $n\in \mathbb{N}$, $\lambda \in \mathbb{C}$, $x\in \mathbb{%
R}^{+}:$%
\begin{equation}
\frac{d^{n}}{dx^{n}}x^{\beta -1}E_{\alpha ,\beta }^{\gamma }\left( \lambda
x^{\alpha }\right) =x^{\beta -n-1}E_{\alpha ,\beta -n}^{\gamma }\left(
\lambda x^{\alpha }\right) .  \label{di}
\end{equation}%
Moreover, it is proved in \cite{MAI} that the Prabhakar function is
completely monotone on $\mathbb{R}^{+}$ (i.e. $f(\cdot )=E_{\alpha ,\beta
}^{\gamma }\left( \cdot \right) $ is infinitely differentiable and such that
$f:(0,+\infty )\rightarrow \mathbb{R}$ with $(-1)^{k}f^{(k)}(x)\geq 0$ for
any $k\in \mathbb{N}$, $x>0$) for the parameters inside the following
ranges: $0<\alpha \leq 1$ and $0<\alpha \gamma \leq \beta \leq 1$.

\section{Holomorphic property on locally convex spaces}

We recall some definitions and theorems on complex analysis in infinite
dimensional convex spaces, for further details see \cite{DIN}. We define
here the complexification of a real Hilbert space as a direct sum $\mathcal{H%
}_{\mathbb{C}}=\mathcal{H}\oplus i \mathcal{H}=\{\xi_1 + i \xi_2 |
\xi_1,\xi_2 \in \mathcal{H}\}$.

\begin{definition}
Given a real Hilbert space $\mathcal{H}$, the scalar product in the
complexification $\mathcal{H}_{\mathbb{C}}$ can be rewritten by using the
bilinear extension of the scalar product in $\mathcal{H}$:
\begin{equation*}
\langle h,g\rangle_{\mathcal{H}_{\mathbb{C}}}=\langle \bar{h},g \rangle_{%
\mathcal{H}} \qquad \text{for } h,g \in \mathcal{H}_{\mathbb{C}}
\end{equation*}
\end{definition}

\begin{definition}
Let be $E$ a vector space on $\mathbb{C}$. $U$ is said ``finitely open" if $%
U \cap F$ is open w.r.t. the Euclidean topology on $F$, for each finite
dimensional subspace $F$ of $E$.
\end{definition}

\begin{definition}
Let $E$ be a vector space on $\mathbb{C}$, $U \subset E$ a finitely open
subset and $F$ a locally convex space. A function $f: U \subset E \to F$ is
said ``Gateaux" or ``G-holomorphic" if $\forall \xi \in U$, $\forall \eta
\in E$ and $\phi \in F^{\prime}$, the function $\mathbb{C} \ni \lambda \to
\phi(f(\xi + \lambda\eta)) \in \mathbb{C}$ is holomorphic on some
neighborhood of $0$ in $\mathbb{C}$.
\end{definition}

Note that we will apply this definition to functions in $\mathbb{C}$, so we
have that $F^{\prime}=\mathbb{C}$, so it is sufficient to check the
holomorphic property on $f$ itself. The following lemma is useful in the
proof of Theorem \ref{ThmLapHolo}, for further details see \cite{KSWY98}.

\begin{lem}
\label{Gholobound} Let $U \subset \mathcal{S}_{\mathbb{C}}$ be open and $f:
U \to \mathbb{C}$. Then $f$ is holomorphic, if and only if it is
G-holomorphic and locally bounded, i.e. each point $\xi \in U$ has a
neighborhood whose image under $f$ is bounded.
\end{lem}

\

\section*{Acknowledgments}  The research by L.B. and L.C. was partially carried over during a visiting period
at the Isaac Newton Institute in Cambridge, whose support is gratefully acknowledged. L.B.
was supported, in particular, by the Kirk Distinguished Fellowship, awarded by
the same institute.

\

\section*{Declarations}

\textbf{Ethical Approval} 
not applicable
\\
\textbf{Competing interests} 
 The authors have no competing interests to declare that are relevant to the content of this article.
 \\
\textbf{Authors' contributions} 
The paper is the outcome of a joint effort and all authors have contributed significantly to every section of the work.
\\
\textbf{Funding} 
not applicable
\\
\textbf{Availability of data and materials}
not applicable

\

\end{document}